\documentclass[12pt,oneside]{amsart}
\usepackage{amsfonts}
\usepackage{amsmath}
\usepackage{amssymb}
\usepackage{mathrsfs}
\usepackage{geometry}
\usepackage{graphicx}
\usepackage{subfigure}
\usepackage{hyperref}
\usepackage{microtype}
\usepackage{enumitem}
\usepackage{geometry}
\usepackage{tikz}
\usepackage{tikz-cd}
\usepackage{color}
\numberwithin{equation}{section}
\allowdisplaybreaks[1]
\usepackage{titletoc}
\usepackage{amsthm}
\usepackage{cite}
 \usepackage{float}

\newcommand{\la}{\lambda}

\newcommand{\ga}{\gamma}

\newcommand{\vp}{\varphi}
\newcommand{\ka}{\kappa}

\newcommand{\R}{\mathbb{R}}

\newcommand{\om}{\omega}

\newcommand{\n}[1]{\Vert #1\Vert }

\newcommand{\bbn}[1]{\Big\Vert #1 \Big \Vert }
\newcommand{\lr}[1]{\left\{ #1 \right\} }
\newcommand{\lrc}[1]{\left[ #1 \right] }
\newcommand{\lrs}[1]{\left( #1 \right) }

\newcommand{\lra}[1]{\langle #1 \rangle }

\newcommand{\abs}[1]{|#1| }
\newcommand{\babs}[1]{\big | #1 \big|}
\newcommand{\bbabs}[1]{\Big | #1 \Big|}
\newcommand{\wt}[1]{\widetilde{#1}}

\newcommand{\pa}{\partial}

\begin{document}

\newtheorem{theorem}{Theorem}[section]
\newtheorem{lemma}[theorem]{Lemma}

\theoremstyle{definition}
\newtheorem{definition}[theorem]{Definition}
\newtheorem{example}[theorem]{Example}
\newtheorem{remark}[theorem]{Remark}

\numberwithin{equation}{section}

\newtheorem{proposition}[theorem]{Proposition}
\newtheorem{corollary}[theorem]{Corollary}
\newtheorem{goal}[theorem]{Goal}

\title[Ill-posedness of the hard-sphere Boltzmann equation]{Sharp $H_{x}^{s}$ Ill-posedness of the Hard-sphere Boltzmann Equation}

\author[X. Chen]{Xuwen Chen}
\address{Department of Mathematics, University of Rochester, Rochester, NY 14627, USA}
\email{xuwenmath@gmail.com}

\author[Y. Guo]{Yan Guo}
\address{Division of Applied Mathematics, Brown University, Providence, RI 02912, USA}
\email{yan\_guo@brown.edu}

\author[S. Shen]{Shunlin Shen}
\address{School of Mathematical Sciences, University of Science and Technology of China, Hefei, 230026, China}
\email{slshen@ustc.edu.cn}

\author[Z. Zhang]{Zhifei Zhang}
\address{School of Mathematical Sciences, Peking University, Beijing, 100871, China}

\email{zfzhang@math.pku.edu.cn}


\begin{abstract}
We investigate the ill-posedness mechanism of the hard-sphere Boltzmann equation in $H_{x}^{s}$ Sobolev space.
Via a direct construction, we prove a strong-weak type ill-posedness result in the low-regularity regime $s<1$, establishing a sharp threshold in connection to the local $s>1$ well-posedness result \cite{CDP19}. Instead of originating
from the large-velocity growth of the collision kernel, this illposedness is generated by the loss
term and dispersive effects. Consequently, we prove a dispersion-driven nonlinear instability mechanism for the hard-sphere Boltzmann equation, and provide a capstone of the ill-posedness series \cite{CH24well, CSZ24well}.
 \end{abstract}
\maketitle
\tableofcontents

\section{Introduction}

We consider the hard-sphere Boltzmann equation
\begin{equation}\label{equ:Boltzmann}
\left\{
\begin{aligned}
\left( \partial_t + v \cdot \nabla_x \right) f (t,x,v) &= Q(f,f),\\
f(0,x,v)&= f_{0}(x,v),
\end{aligned}
\right.
\end{equation}
where $f(t,x,v)$ denotes the distribution function of particles at time $t\geq 0$, position $x\in \mathbb{R}^{3}$, and velocity $v\in \mathbb{R}^{3}$. The collision operator $Q$ is split into a gain term and a loss term:
\begin{align*}
Q(f,g)&=Q^{+}(f,g)-Q^{-}(f,g),
\end{align*}
where
\begin{align*}
Q^{+}(f,g)&=\int_{\mathbb{R}^{3}}\int_{\mathbb{S}^{2}} f(v^{*})g(u^{*}) |(u-v)\cdot \omega|\,dud\omega,\\
Q^{-}(f,g)&=f(v)\int_{\mathbb{R}^{3}}\int_{\mathbb{S}^{2}} g(u) |(u-v)\cdot \omega| \,du d\omega.
\end{align*}
Here, the post-collision velocities $(v^*, u^*)$ are related to the pre-collision velocities $(v, u)$ via
\begin{align*}
u^{*}=u+\lrc{(v-u)\cdot \omega} \omega,\quad v^{*}=v-\lrc{(v-u)\cdot \omega} \omega.
\end{align*}

The hard-sphere Boltzmann equation occupies a key position in the mathematical theory of kinetic gases, provides a bridge between the deterministic description of classical Newtonian particle systems and the macroscopic fluid equations.
It holds profound historical and theoretical significance, serving as the foundational model originally formulated by Maxwell \cite{Max67} and Boltzmann \cite{Bol68} to describe non-equilibrium statistical mechanics.
Furthermore, it remains one of the few settings in which a rigorous derivation from classical particle dynamics has been established. For a comprehensive and detailed discussion, see for example the monograph \cite{Vil02}.

Given its fundamental importance, the Cauchy problem for the hard-sphere Boltzmann equation has attracted a lot of attention, resulting in advances of well-posedness theories. See, for example, \cite{AMUXY13,Ars11,CDP19,DP89,DLX16,DHWY17,DS18,Guo03,Gup03vlasov,Guo04,HJKL24,KS78,IS84,LY04,LY11,Uka74,Uka86}. The early landmark work by DiPerna and Lions \cite{DP89} established the global existence theory of renormalized weak solutions.
Nevertheless, many crucial problems concerning such weak solutions remain open to date.
On the other hand, within the framework of classical well-posedness theory, a natural and important line of research focuses on the well-posedness in $H_x^s$ Sobolev spaces with minimal regularity requirements on initial data.
Exploring well-posedness in the low regularity regime is not only of mathematical interest, but also of direct practical relevance to applications such as numerical simulations. In computational practice, numerical methods often need to handle non-smooth and low regularity data in each iteration, and requires a certain regularity of the solution map to ensure the validity of the algorithm.

The introduction of dispersive estimates has made new progress in this direction.
In fact, the Boltzmann equation shares a number of similarities with the Schr\"{o}dinger equation. Applying the inverse Fourier transform to the Boltzmann equation with respect to the velocity variable yields
\begin{align*}
i\pa_{t}\wt{f}+\nabla_{\xi}\cdot \nabla_{x}\wt{f}=i\mathcal{F}_{v\mapsto \xi}^{-1}\lrc{Q(f,f)},
\end{align*}
where the linear part becomes a hyperbolic Schr\"{o}dinger operator. Moreover, the hierarchy of the Boltzmann equation exhibits a structural similarity to the Gross-Pitaevskii hierarchy associated with the Schr\"{o}dinger equation. To date, substantial progress has been made in the study of quantum many-body hierarchies, particularly concerning space-time collapse estimates (see, e.g., \cite{CHPS15,CP10,Che12,KM08,CH16,CH19,HS19,KSS11}).
By introducing dispersive techniques developed for the study of quantum many-body hierarchy dynamics,\footnote{For further developments utilizing dispersive and harmonic analysis techniques, see, for instance, \cite{BSTW26,CDP21,CSZ23sharp,CSZ26}.}
T. Chen, Denlinger, and Pavlovi$\acute{\text{c}}$ \cite{CDP19} successfully lower the regularity for local well-posedness to $s>1$ for Maxwell
molecules and hard-sphere cases.
On the other hand, many advances have also been made in the study of ill-posedness for dispersive equations. In this direction, see for example \cite{BT06,CCT03,KPV01,MST01,MST02,Nak99}. However, for the Boltzmann equation, despite the dispersive effects present in its linear operator, the potential ill-posed mechanisms are far from obvious due to the complexity of the nonlinear collision kernel.
X. Chen and Holmer \cite{CH24well}, along with follow-up work \cite{CSZ24well}, identified the critical well/ill-posedness threshold to be exactly $s=1$ for both constant kernel and soft potential cases by constructing ill-posedness examples. Nevertheless, the ill-posedness mechanism for the physically significant hard-sphere model remains open.

This opening is not accidental but stems from the distinct mathematical structure of the hard-sphere model. In contrast to the constant kernel and soft potential cases, the hard-sphere collision kernel exhibits growth with respect to the relative velocity (i.e., the large-velocity growth behavior).
Consequently, the nonlinear terms in the equation no longer retain a purely semilinear structure but instead display features of velocity moment loss, which is widely recognized as one of the key difficulties in treating the Cauchy problem for the hard-sphere model (see for example \cite{Vil02}). Indeed,
 the hard-sphere Boltzmann equation exhibits rich and intricate mathematical structures, including the dispersive effects of the linear transport operator, the large-velocity growth induced by the collision kernel, and the nonlocal and asymmetric nature of the gain and loss terms.

The main purpose of this
paper is to explore the interplay of these factors and find the ill-poseness mechanism in $H_{x}^{s}$ Sobolev space (if there is any), and address the problem of finding the well/ill-posedness threshold of the the hard-sphere Boltzmann equation.

\begin{theorem}\label{thm:main theorem}
The hard-sphere Boltzmann equation is ill-posed in $e^{\lra{v}^{2}}L_{v}^{2}H_{x}^{s_{0}}$ for $s_{0}\in [0,1)$,
 in the sense that
 the data-to-solution map is not uniformly continuous.
More specifically, for each $M\gg1 $, there exist a time sequence
$\left\{T_{*}^{M}\right\} _{M}$ satisfying
	$$T_{*}^{M}>0, \quad \lim_{M\to \infty}T_{*}^{M}=0,$$
and two solutions $f^{M}(t)$,  $g^{M}(t)$ in $[0,T_{*}^{M}]$ with
\begin{align*}
\n{e^{\lra{v}^{2}}f^{M}(0)}_{L_{v}^{2}H_{x}^{s_{0}}}\sim \n{e^{\lra{v}^{2}}g^{M}(0)}_{L_{v}^{2}H_{x}^{s_{0}}} \sim 1.
\end{align*}
These solutions are initially close at $t=0$ in the Gaussian-weighted (strong) norm
	\begin{equation*}
		\n{e^{\lra{v}^{2}}\lrs{f^{M}(0)-g^{M}(0)}}_{L_{v}^{2}H_{x}^{s_{0}}}\lesssim \frac{1}{\ln M},
	\end{equation*}
	but become fully separated at $t=T_{*}^{M}$ in the unweighted (weak) norm
	\begin{equation*}
		\n{f^{M}(T_{*}^{M})-g^{M}(T_{*}^{M})}_{L_{v}^{2}H_{x}^{s_{0}}}\sim 1.
	\end{equation*}
	\end{theorem}
The explicit construction of two solutions in Theorem \ref{thm:main theorem} indicates that the origin of its ill-posedness is not the large-velocity growth of the nonlinear term. On the contrary, this ill-posedness is jointly driven by the structure of the loss term and dispersive effects. Therefore, our proof provides a dispersion-driven nonlinear instability
mechanism for the hard-sphere Boltzmann equation.

\begin{remark}
 The validity of the local well-posedness result in the Gaussian-weighted space for the regularity $s>1$ has been proven in \cite{CDP19}. This implies the sharpness for the ill-posedness result, and proves the $s=1$ well/ill-posedness regularity threshold.
 \end{remark}
 \begin{remark}
This is a very strong ill-posedness result. The initial data are arbitrarily close in a strong weighted space, yet they become fully separated in a weak unweighted space, thereby exhibiting a strong-weak type discontinuity.
Phenomena of strong-weak ill-posedness have previously been observed in compressible fluid flows. See, for example, Guo and Tice \cite{GT11}. However, the methods employed therein do not currently appear to overlap with those of the present paper.
Ill-posedness for fluid equations, in the sense of losing uniform continuity of the solution map regarding critical regularity spaces, has also been deeply explored.
For instance, see Bourgain and Pavlovi$\acute{\text{c}}$ \cite{BP08} for the Navier-Stokes equations and Bourgain and Li \cite{BL15,BL15G,BL21} for the incompressible Euler equations.
Despite these macroscopic achievements, ill-posedness results for the Boltzmann equation are remarkably scarce. To the best of our knowledge, Theorem \ref{thm:main theorem} provides the first instance of such a strong-weak instability being established for the Boltzmann equation. A natural problem is whether we can relate this instability to the corresponding phenomena in fluid equations.
\end{remark}

\begin{remark}
Based on the scaling analysis in the $H_{x}^{s}$ space, the scaling-critical index of the Boltzmann equation corresponds to $s=\frac{1}{2}$. This implies that the threshold established in Theorem \ref{thm:main theorem} lies strictly above the value suggested by the scaling argument. Analogous phenomena in the context of dispersive equations, particularly concerning Sobolev indices above the scaling-critical value, have been extensively studied in the work \cite{KPV01} of Kenig, Ponce, and Vega.
Although the linear part of the Boltzmann equation exhibits dispersive effects, the mechanism driving the ill-posedness here is entirely different from that of classical dispersive equations. One primary distinction lies in the time reversibility of the evolution. The intrinsic irreversibility of the hard-sphere Boltzmann equation dictates that its dynamics evolve strictly in a single time direction. In sharp contrast, the time evolution of the Schr\"{o}dinger equation is time-reversible (bidirectional).
Indeed, the ill-posedness mechanism of Theorem \ref{thm:main theorem} is induced by the specific structure of the Boltzmann collision kernel, leading to a norm deflation phenomenon rather than the usual norm inflation.
\end{remark}

\begin{remark}
The proof of Theorem \ref{thm:main theorem} also applies to the case of a constant kernel, yielding a stronger result than that in \cite{CH24well}. However, owing to the essential difference in the structure of the collision kernels, the new counterexample in Theorem \ref{thm:main theorem} differs substantially from the case of soft potential \cite{CSZ24well}, and therefore cannot be directly applied to the latter to yield a strong-weak ill-posedness result.
\end{remark}

\subsection{Outline of the Proof}\label{section:Outline of the Proof}

The origin of the ill-posedness behavior described in \cite{CH24well} is the discrepancy of the dispersive bilinear estimates of the gain term and the loss term. Using the maximizer of the loss term bilinear estimates in the case of a constant
collision kernel, the nonlinear equation effectively reduces to the simplified model
\begin{align}\label{equ:loss term,intro}
\partial_{t}f=-Q^{-}(f,f)\sim -f.
\end{align}
\eqref{equ:loss term,intro} implies a sharp deflation behavior on the solution. Such heuristic analysis indicates the underlying mechanism is actually from harmonic analysis based dispersive equation techniques (maxmizer construction of bilinear estimates) and the nonlinear structure of the Boltzmann equation.

However, in the hard-sphere case, the proof of the ill-posedness mechanism is more like a strong competition between the harmonic analysis techniques and the inherent large-velocity growth of the collision kernel.

We begin the ill-posedness construction of \eqref{equ:Boltzmann} by considering two exact solutions, and decompose them into an approximate part (ansatz) and a correction term as follows
\begin{align*}
f^{M}(t)=&f_{\mathrm{a}}^{M}(t)+f_{\mathrm{c}}^{M}(t),\\
g^{M}(t)=& g_{\mathrm{a}}^{M}(t)+g_{\mathrm{c}}^{M}(t).
\end{align*}
 In order to rigorously prove the ill-posedness result, we need to construct two sequences of solutions
such that the following five conditions are met.
\begin{enumerate}[label=(\alph*), align=left, leftmargin=*]
\item \textbf{Initialization of approximate solutions.}
The approximate solutions start with bounded weighted norm, while the corrections vanish initially. Furthermore, the distance between the two exact solutions is asymptotically negligible:
\begin{equation*}
\left\{
\begin{aligned}
&\n{e^{\lra{v}^{2}}f_{\mathrm{a}}^{M}(0)}_{L_{v}^{2}H_{x}^{s_{0}}}\sim 1,\quad \n{e^{\lra{v}^{2}}g_{\mathrm{a}}^{M}(0)}_{L_{v}^{2}H_{x}^{s_{0}}}\sim 1,\\
&\n{e^{\lra{v}^{2}}f_{\mathrm{c}}^{M}(0)}_{L_{v}^{2}H_{x}^{s_{0}}}= \n{e^{\lra{v}^{2}}g_{\mathrm{c}}^{M}(0)}_{L_{v}^{2}H_{x}^{s_{0}}}=0,\\
&\n{e^{\lra{v}^{2}}\lrs{f^{M}(0)-g^{M}(0)}}_{L_{v}^{2}H_{x}^{s_{0}}}\xrightarrow[M\to \infty]{} 0.
\end{aligned}
\right.
\end{equation*}
Additionally, the unweighted norms remain comparable to the weighted norms:
\begin{align*}
\n{f_{\mathrm{a}}^{M}(0)}_{L_{v}^{2}H_{x}^{s_{0}}}\sim \n{e^{\lra{v}^{2}}f_{\mathrm{a}}^{M}(0)}_{L_{v}^{2}H_{x}^{s_{0}}}\sim 1.
\end{align*}

\item \textbf{Norm deflation property.}
The first approximate solution $f_{\mathrm{a}}^{M}$ exhibits norm deflation property, decay to zero at the terminal time $T_{*}^{M}$:
\begin{equation*}
    \|  e^{\lra{v}^{2}}f_{\mathrm{a}}^{M}(T_{*}^{M}) \|_{L_{v}^{2}H_{x}^{s_{0}}} \xrightarrow[M\to \infty]{} 0.
\end{equation*}

\item \textbf{Stability analysis.}
The profile of the second solution, $g_{\mathrm{a}}^{M}$, remains uniformly close to the initial state of the first approximate solution, $f_{\mathrm{a}}^{M}(0)$, throughout the time interval $[0, T_{*}^{M}]$:
\begin{equation*}
    \sup_{t \in [0, T_{*}^{M}]} \| e^{\langle v \rangle^{2}} ( g_{\mathrm{a}}^{M}(t) - f_{\mathrm{a}}^{M}(0) ) \|_{L_{v}^{2}H_{x}^{s_{0}}} \xrightarrow[M\to \infty]{} 0.
\end{equation*}

\item \textbf{Uniform bounds for correction terms.}
The correction terms must remain asymptotically negligible in the $L_{v}^{2}H_{x}^{s_{0}}$ norm before time $T_{*}^{M}$:
\begin{align*}
  & \sup_{t \in [0, T_{*}^{M}]} \| f_{\mathrm{c}}^{M}(t) \|_{L_{v}^{2}H_{x}^{s_{0}}} \xrightarrow[M\to \infty]{} 0,\\
  & \sup_{t \in [0, T_{*}^{M}]} \| g_{\mathrm{c}}^{M}(t) \|_{L_{v}^{2}H_{x}^{s_{0}}} \xrightarrow[M\to \infty]{} 0.
\end{align*}

\item \textbf{Given short-time Formation.}
The separation time $T_{*}^{M}$ must vanish as the parameter $M$ increases:
\begin{equation*}
    T_{*}^{M} \xrightarrow[M\to \infty]{} 0.
\end{equation*}
\end{enumerate}

Constructing two solutions that simultaneously satisfy the five aforementioned conditions presents a nontrivial challenge. Two primary difficulties arise in this process. The first lies in capturing the ill-posed behavior of the approximate solution, specifically the norm deflation property. The second involves establishing uniform control over the correction terms up to the time when this ill-posed behavior occurs.

Let us focus on the construction of the exact solution $f^{M}=f_{\mathrm{a}}^{M}+f_{\mathrm{c}}^{M}$.
Building on the preliminaries for bilinear estimates of the loss term, we decompose the approximate solution
 $f_{\mathrm{a}}^{M}(t)$ as
\begin{align*}
f_{\mathrm{a}}^{M}(t)=f_{\mathrm{r}}^{M}(t)+f_{\mathrm{b}}^{M}(t),
\end{align*}
where $f_{\mathrm{r}}^{M}(t)$ is the principal term and $f_{\mathrm{b}}^{M}(t)$ acts as a perturbation driving the ill-posed behavior. Moreover, $f_{\mathrm{r}}^{M}(t)$ and $f_{\mathrm{b}}^{M}(t)$ are required to satisfy
\begin{align}
&\partial _{t}f_{\mathrm{b}}^{M}+v\cdot \nabla _{x}f_{\mathrm{b}}^{M}=0,\label{equ:fb,equation,intro}\\
&\partial_{t}f_{\mathrm{r}}^{M}=-Q^{-}(f_{\mathrm{r}}^{M},f_{\mathrm{b}}^{M})=-f_{\mathrm{r}}^{M}A[f_{\mathrm{b}}^{M}],\label{equ:fr,equation,intro}
\end{align}
with
\begin{align*}
 A[f_{\mathrm{b}}^{M}]=\int_{\mathbb{R}^{3}}\int_{\mathbb{S}^{2}} f_{\mathrm{b}}^{M}(u) |(u-v)\cdot \omega| \,du d\omega.
\end{align*}
Then, it is straightforward to verify that the correction term $f_{\mathrm{c}}(t)$ satisfies the nonlinear equation
\begin{equation}\label{equ:correction term,fc,intro}
\left\{
\begin{aligned} \partial_t f_{\mathrm{c}}^{M} + v\cdot \nabla_x f_{\mathrm{c}}^{M} = & \pm Q^\pm(f_{\mathrm{c}}^{M},f_{\mathrm{a}}^{M}) \pm
Q^\pm(f_{\mathrm{a}}^{M},f_{\mathrm{c}}^{M}) \pm Q^\pm(f_{\mathrm{c}}^{M},f_{\mathrm{c}
})-F_{\text{err}}^{M},\\
F_{\mathrm{err}}^{M}=&\pa_{t}f_{\mathrm{a}}^{M}+v\cdot \nabla_{x}f_{\mathrm{a}}^{M}+Q^{-}(f_{\mathrm{a}}^{M},f_{\mathrm{a}}^{M})-Q^{+}(f_{\mathrm{a}}^{M},f_{\mathrm{a}}^{M}).
 \end{aligned}
 \right.
\end{equation}

\textbf{Construction of $f_{\mathrm{b}}^{M}$ and $f_{\mathrm{r}}^{M}$.}
The key novelty of the paper is the delicate construction of $f_{\mathrm{b}}^{M}$ and $f_{\mathrm{r}}^{M}$, which ensures that the approximate solution $f_{\mathrm{a}}^{M}$ exhibits a norm deflation property, while also providing uniform control of the subsequent correction term $f_{\mathrm{c}}^{M}$.
 In particular, the construction of
$f_{\mathrm{b}}^{M}$ relies on a lattice system of nearly uniformly distributed lattice points $\lr{e_{i,j}}$ on the unit hemisphere, inspired by the variant of $l^{2}$-decoupling implemented \cite{CSZ26}.
  Specifically, on the unit sphere, we set
 \begin{align}\label{equ:3d,eij,intro}
 e_{i,j}=\lrs{\sin(\frac{\pi j}{2M}) \cos( \frac{2\pi i}{j}),\sin(\frac{\pi j}{2M})\sin(\frac{2\pi i}{j}),\cos (\frac{\pi j}{2M})},\quad 1\leq i\leq j\leq M.
 \end{align}
\begin{figure}[H]
  \centering
  \includegraphics[width=6cm]{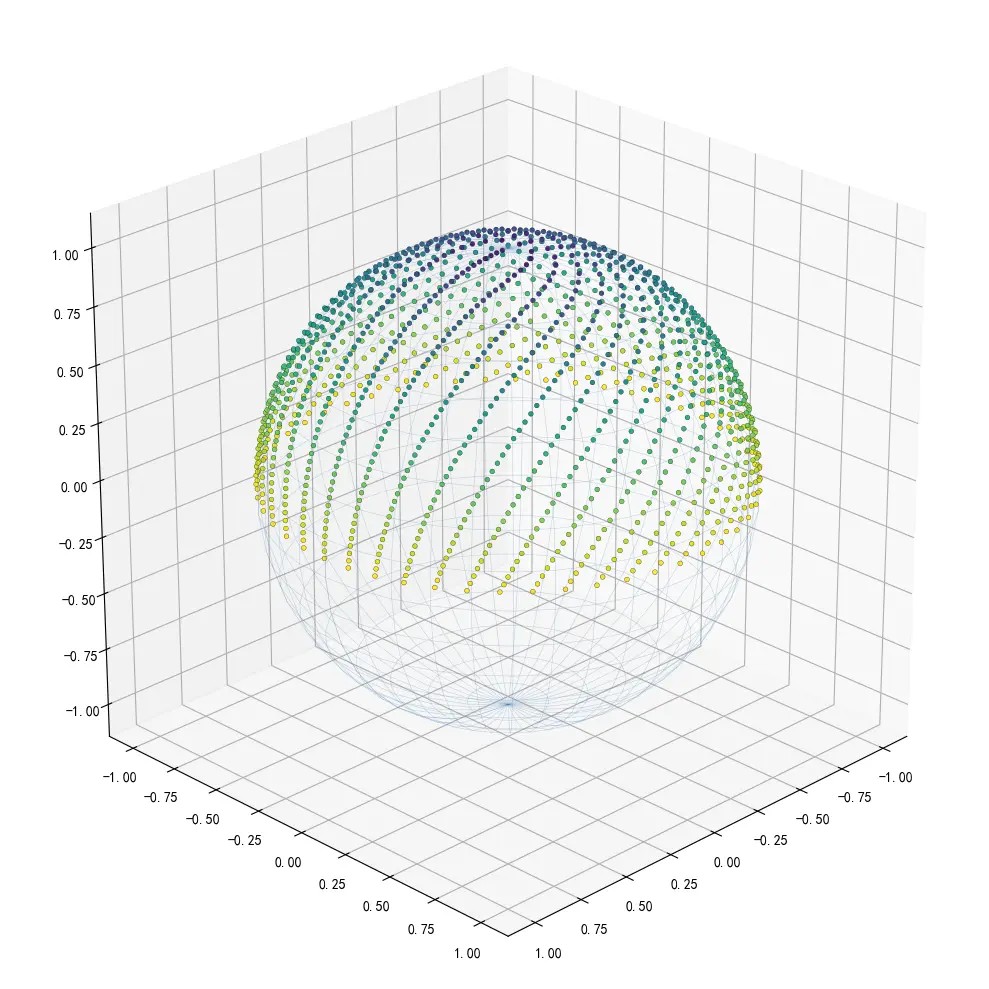}
\caption{$e_{i,j}$ points on the unit sphere for $M=50$}
\label{fig:sphere}
\end{figure}
As illustrated in Figure \ref{fig:sphere}, these points are roughly uniformly spaced over the upper hemisphere. They exhibit rotational periodicity in the horizontal directions but lack symmetry in the vertical
$z$-direction. Although this absence of full symmetry may introduce certain difficulties in analysis and summation estimates, it does not fundamentally hinder our approach. In fact, our construction does not heavily rely on such full symmetry. Rather, it enjoys considerable freedom and flexibility. Indeed, the construction is not strictly confined to the hemisphere and can be naturally adapted or restricted to smaller spherical domains as needed.

 To formalize the construction, let $P_{e_{i,j}}$ denote the orthogonal projection onto the one-dimensional subspace spanned by $e_{i,j}$, and let $P_{e_{i,j}}^{\perp}$ be the orthogonal projection onto the orthogonal complement $\lr{e_{i,j}}^{\perp}$.
The function $f_{\mathrm{b}}^{M}$ is then constructed as a superposition of tubes along lattice points in different directions
 \begin{equation}\label{equ:fb,t,intro}
f_{\mathrm{b}}^{M}(t,x,v)=M^{1-s}\sum_{1\leq i\leq j\leq M}K_{i,j}(x-vt)I_{i,j}(v),
\end{equation}
where
\begin{align*}
K_{i,j}(x)&=\chi (MP_{e_{i,j}}^{\perp }x)\chi (P_{e_{i,j}}x), \\
I_{i,j}(v)&=\chi(MP_{e_{i,j}}^{\perp }v)\chi (10P_{e_{i,j}}(v-e_{i,j})).
\end{align*}
Here, $\chi(x)$ is a smooth cutoff function supported in the unit ball $B(0,1)$.
The regularity parameter $s$ is chosen specifically to satisfy $\max\lr{s_{0},\frac{3}{4}}<s<1$.

\begin{minipage}{0.5\textwidth}
\begin{figure}[H]
  \centering
  \includegraphics[width=6cm]{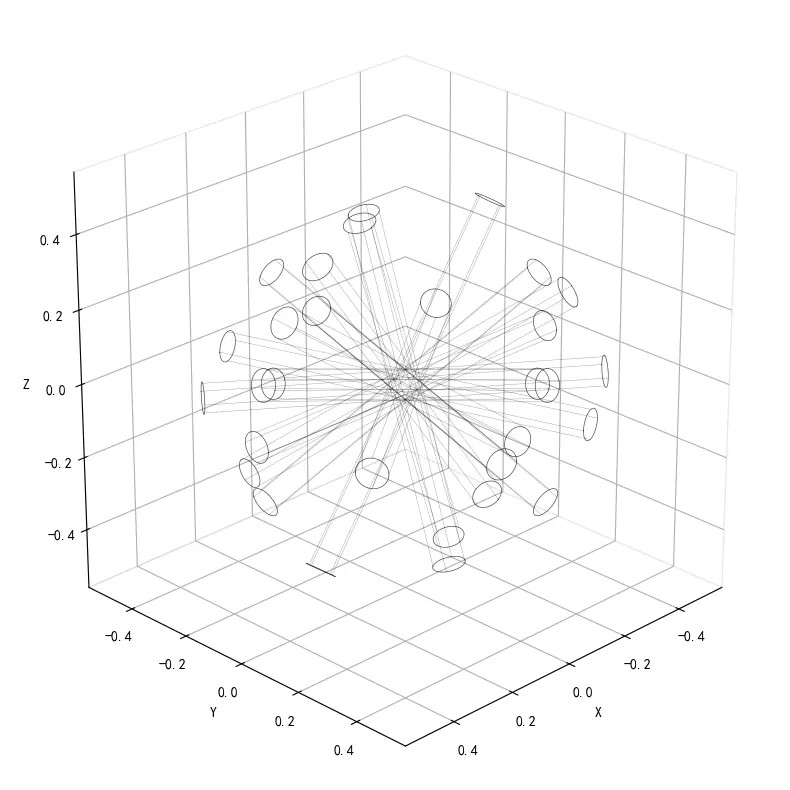}
\caption{$K_{i,j}(x)$ for $M=5$}
\label{fig:tube}
\end{figure}
\end{minipage}
\begin{minipage}{0.5\textwidth}
\begin{figure}[H]
  \centering
  \includegraphics[width=6cm]{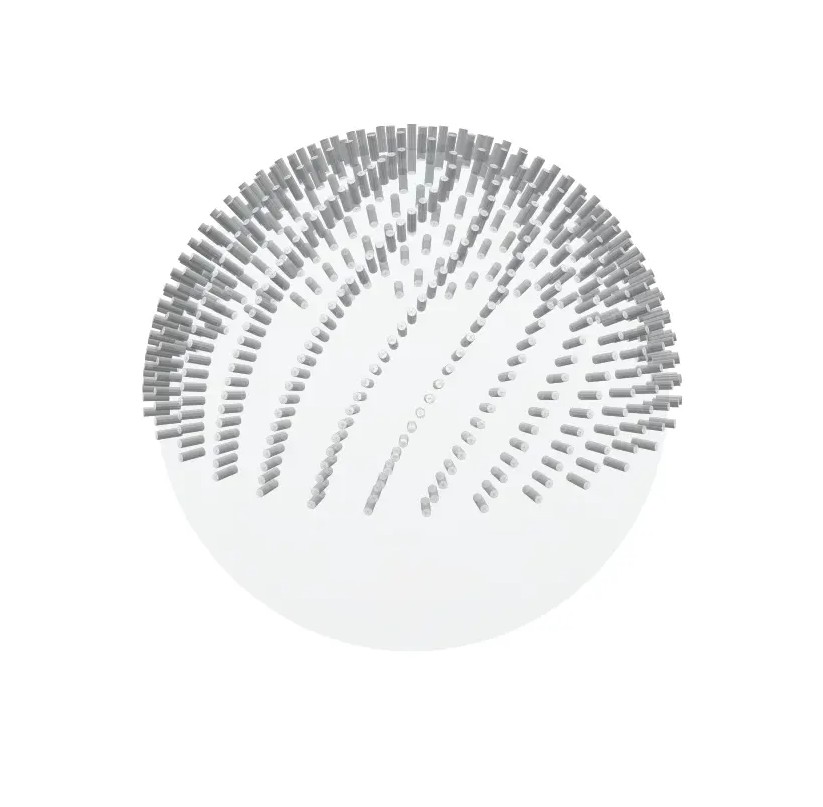}
\caption{$I_{i,j}(v)$ for $M=30$}
\label{fig:Iij}
\end{figure}
\end{minipage}
\vspace{0.5em}

The principal term $f_{\mathrm{r}}^{M}$ is initially supported near the origin
$$f_{\mathrm{r}}^{M}(0,x,v)=M^{\frac{3}{2}-s_{0}}N^{\frac{3}{2}} \chi (100Mx)\chi(Nv).$$
By the equation \eqref{equ:fr,equation,intro}, the solution $f_{\mathrm{r}}(t)$ can be expressed as
\begin{equation}\label{equ:fr,intro}
	f_{\mathrm{r}}^{M}(t,x,v)=M^{\frac{3}{2}-s_{0}}N^{\frac{3}{2}}\exp \left[
-\int_{0}^{t}A[f_{\mathrm{b}}^{M}](\tau,x,v)\,d\tau \right] \chi (100Mx)\chi (Nv).
\end{equation}

One key feature of $f_{\mathrm{b}}^{M}$ and $f_{\mathrm{r}}^{M}$ is the uniform compact support condition that
\begin{equation}\label{equ:uniform compact,intro}
\left\{
\begin{aligned}
\operatorname{Supp}f_{\mathrm{r}}^{M}(t)\subset B(0,1),\\
\operatorname{Supp}f_{\mathrm{b}}^{M}(t)\subset B(0,1).
\end{aligned}
\right.
\end{equation}
This uniform localization is essential for overcoming the growth of the Gaussian weight $e^{\lra{v}^{2}}$,
which is an analytic challenge inherent to the hard-sphere model.
More importantly, the specific geometric features of our construction are equally crucial for achieving the ill-posedness behaviour of solutions.
In what follows, we outline the key ideas, techniques, and difficulties involved in the proof.

\begin{enumerate}[wide=0pt, label=(\arabic*), align=left]

\setlength{\parindent}{2em}
\item \textbf{Dispersive techniques v.s. compact geometry}.
For the hard-sphere Boltzmann equation, the feature of the large-velocity growth is a main obstacle.
A seemingly available solution is to impose a uniform localization constraint \eqref{equ:uniform compact,intro}.

However, this constraint is quite restrictive and significantly limits the construction of approximate solutions.
In \cite{CH24well,CSZ24well}, the maximizer of the bilinear estimates for the loss term is built from a superposition of elongated tubes, which naturally allows for two free parameters. The condition \eqref{equ:uniform compact,intro}, by contrast, permits only a single-parameter construction.
Due to this limitation, the approximate solutions provided in \cite{CH24well,CSZ24well} are not applicable.

In other words, the construction of $f_{\mathrm{b}}^{M}$, along with the underlying dispersive techniques and harmonic analysis, is subject to this uniform localization constraint. Nevertheless, within the setting of compact geometry, we establish a uniform gap between distinct directions of lattice points $e_{i,j}$, summation bounds for the tubes $K_{i,j}(x)$, and the disjointness of the supports of the velocity-localized functions
$I_{i,j}(v)$. Specifically, we have
\begin{align}
|e_{i_{1},j_{1}}-e_{i_{2},j_{2}}|\geq& \frac{1}{M},\quad \text{for $(i_{1},j_{1})\neq (i_{2},j_{2})$},\label{equ:gap estimate,eij,intro}\\
\sum_{i,j}|\nabla_{x}^{k}K_{i,j}(x)|\lesssim& M^{2+k}\lrs{\frac{1}{M|x|+1}}^{2}\chi(x),\quad k=0,1,\label{equ:Kij,sum,intro}\\
\operatorname{Supp}I_{i_{1},j_{1}}\bigcap \operatorname{Supp}I_{i_{2},j_{2}}=&\emptyset,\quad
\text{for $(i_{1},j_{1})\neq (i_{2},j_{2})$},\\
\sum_{i,j}I_{i,j}(v)\leq& 1_{\lr{\frac{3}{4}\leq|v|\leq \frac{5}{4}}}(v).\label{equ:Iij,sum,intro}
\end{align}
 These estimates capture the geometric features of $f_{\mathrm{b}}^{M}$ and serve as the foundation for the subsequent analysis.

\item \textbf{Forward solving of the approximate solution due to velocity weight loss.}
Previous studies \cite{CH24well,CSZ24well} on the soft potential and constant kernel cases proved the existence of initial data via backward solving.
However, due to the irreversibility caused by the nature of the hard-sphere model, such an existence strategy is no longer applicable. Therefore, we must directly construct the approximate solution and solve the correction equation \eqref{equ:correction term,fc,intro} by forward time iteration.

We now briefly explain how our forward-in-time constructions of $f_{\mathrm{b}}^{M}$ and $f_{\mathrm{r}}^{M}$ indeed directly induce the desired ill-posed behavior.
First, the construction of $f_{\mathrm{b}}^{M}$ ensures that its profile structure remains essentially invariant over an
$O(1)$ time scale.
Based on \eqref{equ:Kij,sum,intro}, \eqref{equ:Iij,sum,intro}, and the specific choice of the regularity parameter $s$, the perturbation $f_{\mathrm{b}}^{M}$ satisfies the following uniform-in-time decay estimate
\begin{align}\label{equ:critical norm estimate for fb,intro}
\n{e^{\lra{v}^{2}}f_{\mathrm{b}}^{M}(t)}_{L_{v}^{2}H_{x}^{s_{0}}}\lesssim & M^{s_{0}-s}\ll 1,\quad \text{for $t\leq \frac{1}{4}.$}
\end{align}
More importantly, the particular construction of $f_{\mathrm{b}}^{M}$
concentrates most of the mass near the center of the ball, which is the key factor generating the ill-posed behavior. Specifically, within the support of
$f_{\mathrm{r}}^{M}$, namely $\lr{(x,v):|x|\leq \frac{1}{100M},|v|\leq N^{-1}}$, the term $A[f_{\mathrm{b}}^{M}]$ exhibits significant growth
\begin{align*}
A[f_{\mathrm{b}}^{M}](t,x,v)\sim \int_{\R^{3}}|u-v|f_{\mathrm{b}}^{M}(t,x,u)du \sim M^{-1-s}\sum_{i,j}K_{i,j}(x)\sim M^{1-s}\gg 1.
\end{align*}
This, together with the expression \eqref{equ:fr,intro} for $f_{\mathrm{r}}^{M}$, allows us to further establish that the principal term contracts rapidly over a short time interval
\begin{align}
\n{e^{\lra{v}^{2}}f_{\mathrm{r}}^{M}(t)}_{L_{v}^{2}H_{x}^{s_{0}}}\lesssim&  \exp[-t M^{1-s}].\label{equ:critical norm estimate for fr,upper bound,intro}
\end{align}
Consequently, by combining \eqref{equ:critical norm estimate for fb,intro} and \eqref{equ:critical norm estimate for fr,upper bound,intro}, we obtain the norm deflation property for the approximate solution
\begin{equation}\label{equ:critical norm estimate for fa,intro}
\left\{
\begin{aligned}
	\n{e^{\lra{v}^{2}}f_{\mathrm{a}}^{M}(0)}_{L_v^{2} H_x^{s_{0}}}\sim&  1,\\
	\n{e^{\lra{v}^{2}}f_{\mathrm{a}}^{M}(T_{*})}_{L_v^{2} H_x^{s_{0}}}\lesssim& \frac{1}{\ln M},
\end{aligned}
\right.
\end{equation}
where $T_{*}^{M}=M^{s-1}(\ln \ln M)$ becomes arbitrarily short in the limit, i.e.,
\begin{align*}
\lim_{M\to \infty}T_{*}^{M}=0.
\end{align*}

\item \textbf{Space-time estimates in the Gaussian-weighted $Z$-norm}.
To perturb the approximate solution into an exact solution to \eqref{equ:Boltzmann}, we need to establish a well-posedness theory for the correction equation \eqref{equ:correction term,fc,intro} and prove
 \begin{align}\label{equ:fc,decay,intro}
 \sup_{t \in [0, T_{*}^{M}]} \| f_{\mathrm{c}}^{M}(t) \|_{L_{v}^{2}H_{x}^{s_{0}}} \xrightarrow[M\to \infty]{} 0.
 \end{align}
Usually, one would expect a time-independent bilinear estimate of the form
\begin{align*}
\n{Q^{\pm}(f,g)}_{Z}\lesssim \n{f}_{Z}\n{g}_{Z}.
\end{align*}
Unfortunately, for the hard-sphere collision kernel, due to the loss of velocity weight, such an estimate generally fails no matter which $Z$-norm is chosen. To overcome this difficulty, we instead establish a collision estimate in the time-integrated form, using a Gaussian-weighted $Z$-norm
    \begin{align}\label{equ:bilinear estimate,space-time,intro}
&\bbn{e^{(1- t)\lra{v}^{2}}\int_{t_{0}}^{t}e^{-(t-\tau)v\cdot \nabla_{x}}Q^{\pm}(f,g)(\tau)d\tau}_{Z}\\
\lesssim & \lrs{t-t_{0}}^{\frac{1}{2}}\n{e^{(1- \tau)\lra{v}^{2}}f}_{L_{\tau}^{\infty}(t_{0},t;Z)}
\n{e^{(1- \tau)\lra{v}^{2}}g}_{L_{\tau}^{\infty}(t_{0},t;Z)}.\notag
\end{align}
Here, the
$Z$-norm requires a delicate choice. On the one hand, it must be sufficiently strong to close the bilinear space-time estimates. On the other hand, it needs to be as weak as possible to ensure the
$Z$-norm of the error terms remains uniformly controllable in parameter $M$ as in \eqref{equ:Z-norm,intro}, because higher regularity norms of $f_{\mathrm{a}}^{M}$ and the error terms grow quickly in $M$. Specifically, the
$Z$-norm we construct is a composite of the following four norms
\begin{align}\label{equ:Z-norm,intro}
\n{f}_{Z}:=M^{s_{0}}\n{f}_{L_{v}^{2,1}L_{x}^{2}}+M^{s_{0}-1}\n{\nabla_{x}f}_{L_{v}^{2,1}L_{x}^{2}}+\n{f}_{L_{v}^{1,1}L_{x}^{\infty}}+
M^{-1}\n{\nabla_{x}f}_{L_{v}^{1,1}L_{x}^{\infty}},
\end{align}
where $\n{f}_{L_{v}^{p,1}}:=\n{\lra{v}f}_{L_{v}^{p}}$.

Although space-time estimate \eqref{equ:bilinear estimate,space-time,intro} directly implies some local well-posedness, the resulting lifespan $T_{0}$
is typically very short. This short time scale depends primarily on the $Z$-norm of the approximate solution
$f_{\mathrm{a}}^{M}$ and the prefactor $(t-t_{0})^{\frac{1}{2}}$ appearing in estimate \eqref{equ:bilinear estimate,space-time,intro}. More precisely,
\begin{align*}
T_{0}\sim \n{f_{\mathrm{a}}^{M}}_{Z}^{2}\sim M^{2(s-1)}\ll T_{*}^{M}=M^{s-1}(\ln\ln M).
\end{align*}
Our main objective is to extend the lifespan of the solution up to the time at which ill-posed behaviour emerges. To achieve this, we employ an iterative scheme that extends the existence interval to the target time $T_{*}^{M}$.

However, this strategy is not straightforward. In each iteration, the control of the correction term grows exponentially. Consequently, the number of iterations is strictly constrained by uniform-in-time decay estimates of the error term $F_{\mathrm{err}}^{M}$. At the same time, the maximal single-step time scale
$T_{0}$ is so small that it prevents us from reaching the critical time $T_{*}^{M}$ within the limited number of iterations allowed.

To circumvent this time-scale obstruction, the idea is to introduce a rate parameter $\kappa$
that provides additional smallness. Based on this, we establish the following parameter-dependent estimate
  \begin{align}\label{equ:bilinear estimate,space-time,kappa,intro}
&\bbn{e^{(1-\kappa t)\lra{v}^{2}}\int_{t_{0}}^{t}e^{-(t-\tau)v\cdot \nabla_{x}}Q^{\pm}(f,g)(\tau)d\tau}_{Z}\\
\lesssim & \lrs{\frac{t-t_{0}}{\kappa}}^{\frac{1}{2}}\n{e^{(1-\kappa \tau)\lra{v}^{2}}f}_{L_{\tau}^{\infty}(t_{0},t;Z)}
\n{e^{(1-\kappa \tau)\lra{v}^{2}}g}_{L_{\tau}^{\infty}(t_{0},t;Z)}.\notag
\end{align}
Then, by setting $\kappa=\frac{1}{T_{*}^{M}}$, we can take the step size to be $T_{0}\sim \n{f_{\mathrm{a}}^{M}}_{Z}\sim M^{s-1}$,
which suffices to meet the minimal step size required for our iteration.

\item \textbf{$Z$-norm decay estimates for the error terms.} Finally, in order to obtain the uniform-in-time decay estimate of $f_{\mathrm{c}}^{M}$ stated in \eqref{equ:fc,decay,intro},
we must prove that the error term $F_{\mathrm{err}}^{M}$ enjoys a uniform-in-time decay estimate as well. Specifically,
\begin{equation}\label{equ:Z-norm,error,intro}
\sup_{0\leq t_{0}\leq t\leq T_{*}^{M}}\bbn{e^{\lra{v}^{2}}\int_{t_{0}}^{t}e^{-(t-\tau)v\cdot \nabla _{x}}F_{\text{err}
}^{M}(\tau)\,d\tau}_{Z}\lesssim M^{-\delta}.
\end{equation}
Estimating the
$Z$-norm of the error term is the most intricate part of the entire proof. Since the
$Z$-norm \eqref{equ:Z-norm,intro} itself consists of four sub-norms and there are eight error terms in \eqref{equ:correction term,fc,intro}, there are a total of 32 fine error bounds. On the other hand, these estimates rely heavily on the specific construction of the approximate solution, the aforementioned geometric estimates. In fact, \eqref{equ:Z-norm,intro} involves a super-critical norm $L_{x}^{\infty}$, which scales like $\dot{H}_{x}^{\frac{3}{2}}$. To ensure that the error bound in \eqref{equ:Z-norm,error,intro} decays in $M$, delicate dispersive techniques are required.

For the time integral of the error term, a direct estimation strategy exploits the shortness of the time interval to provide a smallness factor
\begin{equation}\label{equ:Ferr,intro}
	\bbn{e^{\lra{v}^{2}} \int_{t_{0}}^{t}e^{-(t-\tau)v\cdot \nabla _{x}}F_{\text{err}}^{M}(\tau)\,d\tau}_{Z}\lesssim T_{*}^{M}\n{e^{\lra{v}^{2}}F_{\text{err}}^{M}}_{L_{t}^{\infty }Z}.
\end{equation}
Most terms within the error term $F_{\mathrm{err}}^{M}$ can first be handled via
$L^{p}$-type estimates for the Boltzmann collision kernel.
Moreover, effective control is achieved by leveraging the uniform compact support which serves to absorb the Gaussian weight $e^{\lra{v}^{2}}$,
the Sobolev norm estimates of $f_{\mathrm{r}}^{M}$ and $f_{\mathrm{b}}^{M}$, and the smallness of the time scale $T_{*}^{M}$.

 However, a notable exception arises in treating the $L_{v}^{1,1}L_{x}^{\infty}$ norm of $Q^{\pm}(f_{\mathrm{b}}^{M},f_{\mathrm{b}}^{M})$.
For brevity, we introduce the following notation
\begin{equation*}
	D^{\pm}=\int_{t_{0}}^{t}e^{-(t-\tau)v\cdot \nabla
		_{x}}Q^{\pm}(f_{\mathrm{b}}^{M},f_{\mathrm{b}}^{M})(\tau)d\tau.
\end{equation*}
A direct application of the estimate \eqref{equ:Ferr,intro} yields
$$\n{D^{\pm}}_{L_{v}^{1,1}L_{x}^{\infty
}}\lesssim T^{M}_{*}\n{Q^{\pm}(f_{\mathrm{b}}^{M},f_{\mathrm{b}}^{M})}_{L_{v}^{1,1}L_{x}^{\infty
}}\lesssim T^{M}_{*}\n{f_{\mathrm{b}}^{M}}_{Z}^{2}\lesssim M^{1-s}(\ln\ln M)\to \infty,$$
which leads to a large growth term. To circumvent this, our strategy is to employ time-integration techniques from dispersive equations to achieve a substantial gain in the estimates.

Notably, the function $f_{\mathrm{b}}^{M}$ itself is constructed as a sum of numerous tubes given in \eqref{equ:fb,t}
and these tubes are mutually coupled through the collision kernel. This significantly increases the difficulty and complexity of the analysis. For instance,
\begin{align*}
D^{+} =&M^{2-2s}\sum_{i_{1},j_{1}}\sum_{i_{2},j_{2}}\int_{t_{0}}^{t}e^{-(t-\tau)v\cdot \nabla_{x}}Q^{+}(K_{i_{1},j_{1}}(x- v\tau )I_{i_{1},j_{1}}(v),K_{i_{2},j_{2}}(x-v\tau)I_{i_{2},j_{2}}(v))d\tau.
\end{align*}
Estimating $D^{+}$ requires careful geometric estimates on the decomposition of the hemisphere.
We need to analyze the effect of $Q^{+}$ on the hemisphere grid, and the measures of the localized projections of $D^+$ induced by the time integration. In fact, since $Q^{+}$ generates the collision projection $P_{e_{i_{1},j_{1}}}^{\perp}P_{\omega
}(u-v)$, an angular decomposition of $\omega$ is also needed.
Furthermore, a key part of the proof lies in properly handling the geometric and summation estimates over the lattice system. A typical estimate needed is of the following form
\begin{align*}
\sum_{(i_{1},j_{1})\neq (i_{2},j_{2})} \frac{1}{|e_{i_{1},j_{1}}-e_{i_{2},j_{2}}|}\lesssim M^{4}\ln M.
\end{align*}
In summary, the treatment of this proof requires various analytical techniques, including localization via time integration, projection measures, angular decomposition, lattice counting, and geometric estimates.
\end{enumerate}

\section{Norm Deflation Property for the Approximate Solution} \label{sec:Bounds on the approximation solution}
In this section, we aim to construct approximate solutions satisfying the norm deflation property. As discussed in Section \ref{section:Outline of the Proof}, we decompose the approximate solution $f_{\mathrm{a}}$ into two components
\begin{align*}
f_{\mathrm{a}}(t)=f_{\mathrm{r}}(t)+f_{\mathrm{b}}(t),
\end{align*}
where $f_{\mathrm{b}}$ satisfies the linear transport equation
\begin{align}\label{equ:fb equation}
\partial _{t}f_{\mathrm{b}}+v\cdot \nabla _{x}f_{\mathrm{b}}=0,
\end{align}
and $f_{\mathrm{r}}$ satisfies the drift-free linearized loss-term Boltzmann equation
\begin{align}\label{equ:fr,loss-term}
\partial_{t}f_{\mathrm{r}}=-Q^{-}(f_{\mathrm{r}},f_{\mathrm{b}}).
\end{align}
Here, for notational simplicity, we omit the superscript parameter $M$ throughout.

For the construction of $f_{\mathrm{b}}$, we rely on a lattice system that is roughly uniformly distributed on the unit hemisphere. Specifically, we consider the lattice points on the unit sphere defined by
 \begin{align}\label{equ:3d,eij}
 e_{i,j}=\lrs{\sin(\frac{\pi j}{2M}) \cos( \frac{2\pi i}{j}),\sin(\frac{\pi j}{2M})\sin(\frac{2\pi i}{j}),\cos (\frac{\pi j}{2M})},\quad 1\leq i\leq j\leq M.
 \end{align}
For an intuitive illustration, see Figure \ref{fig:sphere}. These points are roughly uniformly spaced across the upper hemisphere, with the total number of points scaling as order $M^{2}$. As introduced in \eqref{equ:fb,t,intro},
 the perturbation term $f_{\mathrm{b}}$ is constructed as a superposition of tubes aligned along these lattice points in various directions, given by
 \begin{equation}\label{equ:fb,t}
f_{\mathrm{b}}(t,x,v)=M^{1-s}\sum_{1\leq i\leq j\leq M}K_{i,j}(x-vt)I_{i,j}(v),
\end{equation}
where
\begin{align*}
K_{i,j}(x)&=\chi (MP_{e_{i,j}}^{\perp }x)\chi (P_{e_{i,j}}x), \\
I_{i,j}(v)&=\chi(MP_{e_{i,j}}^{\perp }v)\chi (10P_{e_{i,j}}(v-e_{i,j})).
\end{align*}

The principal term $f_{\mathrm{r}}$ is initially supported near the origin
$$f_{\mathrm{r}}(0,x,v)=M^{\frac{3}{2}-s_{0}}N^{\frac{3}{2}} \chi (100Mx)\chi(Nv).$$
By the equation \eqref{equ:fr,loss-term}, the solution $f_{\mathrm{r}}(t)$ can be expressed as
\begin{equation}\label{equ:fr}
	f_{\mathrm{r}}(t,x,v)=M^{\frac{3}{2}-s_{0}}N^{\frac{3}{2}}\exp \left[
-\int_{0}^{t}A[f_{\mathrm{b}}](\tau,x,v)\,d\tau \right] \chi (100Mx)\chi (Nv),
\end{equation}
where the term $A[f_{\mathrm{b}}]$ is given by (up to a constant)
\begin{equation}\label{equ:Afb}
A[f_{\mathrm{b}}](t,x,v)\sim \int_{\R^{3}}|u-v|f_{\mathrm{b}}(t,x,u)du.
\end{equation}

\paragraph{\textbf{Parameter selection}.}
Throughout this paper, the parameters are chosen as follows:
\begin{align}
	&M\gg 1, \quad  N=  M^{10},\label{equ:condition,parameters}\\
&\max\lr{s_{0},\frac{3}{4}}<s<1, \label{equ:condition,s,s0}\\
&T_{*}=M^{s-1}(\ln\ln M).\label{equ:T*}
\end{align}

We proceed by establishing the geometric estimates and fundamental properties of the requisite lattice system in Section \ref{sec:Geometry Estimates of Directional Projections}. Subsequently, in Section \ref{section:Sobolev Norm Estimates for Approximate Solutions}, we derive the Sobolev norm estimates for the approximate solutions and prove the norm deflation property for $f_{\mathrm{a}}$.
\subsection{Geometry Estimates of Directional Projections}\label{sec:Geometry Estimates of Directional Projections}
 We provide geometric estimates for a family of directions ${e_{i,j}}$ distributed on the upper hemisphere. We establish a uniform gap between distinct directions, summation bounds for the tubes $K_{i,j}$, and the disjointness of the velocity-localized functions $I_{i,j}$. These estimates are essential in subsequent arguments.
\begin{lemma}[Separation of directions and summation bounds]\label{lemma:eij,property}
 The following estimates hold
\begin{align}
|e_{i_{1},j_{1}}-e_{i_{2},j_{2}}|\geq& \frac{1}{M},\quad \text{for $(i_{1},j_{1})\neq (i_{2},j_{2})$},\label{equ:gap estimate,eij}\\
\sum_{i,j}|\nabla_{x}^{k}K_{i,j}(x)|\lesssim& M^{2+k}\lrs{\frac{1}{M|x|+1}}^{2}\chi(x),\quad k=0,1,\label{equ:Kij,sum}\\
\sum_{i,j}I_{i,j}(v)\leq& 1_{\lr{\frac{3}{4}\leq|v|\leq \frac{5}{4}}}(v).\label{equ:Iij,sum}
\end{align}
\end{lemma}
\begin{proof}
\textbf{1. Geometric preliminaries.}

We begin by establishing geometric bounds on the separation between points. Let us define two points in spherical coordinates
\begin{align*}
& p_{1}=\left(\sin \theta_1 \cos \varphi_1, \sin \theta_1 \sin \varphi_1, \cos \theta_1\right), \ \theta_{1}\in[0,\frac{\pi}{2}],\ \vp_{1}\in[0,2\pi],\\
& p_{2}=\left(\sin \theta_2 \cos \varphi_2, \sin \theta_2 \sin \varphi_2, \cos \theta_2\right), \
\theta_{2}\in[0,\frac{\pi}{2}],\ \vp_{2}\in[0,2\pi].
\end{align*}
Let $\gamma$ denote the angle between $p_{1}$ and $p_{2}$. By the law of cosines, we have
\begin{align}
|p_{1}-p_{2}|=&\sqrt{2-2\cos \ga},\label{equ:p1p2,distance}
\end{align}
where
\begin{align}
\cos \ga=&\cos \theta_1 \cos \theta_2+\sin \theta_1 \sin \theta_2 \cos \left(\varphi_1-\varphi_2\right)\notag\\
=&\cos (\theta_{1}-\theta_{2})+\sin \theta_1 \sin \theta_2 \lrs{\cos \left(\varphi_1-\varphi_2\right)-1}.\notag
\end{align}
These expressions yield two useful inequalities. First,
\begin{align}
\cos \ga\leq& \cos (\theta_{1}-\theta_{2}).\label{equ:gamma,theta12}
\end{align}
Second, assuming without loss of generality that $0 \leq \theta_{1} \leq \theta_{2} \leq \frac{\pi}{2}$, we have
\begin{align}
\cos \ga\leq& 1+(\sin\theta_{1})^{2}\lrs{\cos \left(\varphi_1-\varphi_2\right)-1},\quad 0\leq \theta_{1}\leq \theta_{2}\leq \frac{\pi}{2}.\label{equ:gamma,varphi12}
\end{align}
Substituting \eqref{equ:gamma,theta12} into \eqref{equ:p1p2,distance}, we obtain a lower bound dependent on the polar angles
\begin{align}\label{equ:p1p2,lower,bound,theta}
|p_{1}-p_{2}|\geq \sqrt{2-2\cos (\theta_{1}-\theta_{2})}\geq 2\bbabs{\sin\frac{\theta_{1}-\theta_{2}}{2}}\geq \frac{2}{\pi}|\theta_{1}-\theta_{2}|.
\end{align}
Similarly, inserting \eqref{equ:gamma,varphi12} into \eqref{equ:p1p2,distance} provides a bound involving the azimuthal angles
\begin{align}
|p_{1}-p_{2}|\geq& \sqrt{2-2(1+(\sin\theta_{1})^{2}\lrs{\cos \left(\varphi_1-\varphi_2\right)-1)}}
\label{equ:p1p2,lower,bound,varphi}\\
=&\sin \theta_{1}\sqrt{2-2\cos(\vp_{1}-\vp_{2})}\geq \frac{4}{\pi^{2}} \theta_{1} |\vp_{1}-\vp_{2}|.\notag
\end{align}

\textbf{2. Proof of the directional separation estimate \eqref{equ:gap estimate,eij}.}

We now apply these geometric bounds to the discrete points $e_{i,j}$. Recall the discretization parameters associated with $e_{i,j}$,
specifically $\theta=\frac{\pi j}{2M}$, and $\vp=\frac{2\pi i}{j}$.

We consider first the case $j_{1}\neq j_{2}$. Using \eqref{equ:p1p2,lower,bound,theta}, we get
\begin{align}
|e_{i_{1},j_{1}}-e_{i_{2},j_{2}}|\geq \frac{2}{\pi}\frac{\pi |j_{1}-j_{2}|}{2M}\geq \frac{|j_{1}-j_{2}|}{M}\geq \frac{1}{M}.\label{equ:gap estimate,eij,j1neqj2}
\end{align}

Next, we consider the case $j_{1}=j_{2}=j$ with $i_{1}\neq i_{2}$.
 Applying \eqref{equ:p1p2,lower,bound,varphi}, we obtain
\begin{align}
|e_{i_{1},j_{1}}-e_{i_{2},j_{2}}|\geq  \frac{4}{\pi^{2}}\frac{\pi j}{2M} \frac{2\pi |i_{1}-i_{2}|}{j}=\frac{8}{M}.\label{equ:gap estimate,eij,j1=j2}
\end{align}
Combining \eqref{equ:gap estimate,eij,j1neqj2} and \eqref{equ:gap estimate,eij,j1=j2} yields the desired separation estimate \eqref{equ:gap estimate,eij}.

\textbf{3. Proof of the summation bound estimates \eqref{equ:Kij,sum}.}

Given that the $x$-derivative $\nabla_{x}$ acting on $K_{i,j}(x)$ produces a factor $M$,
it suffices to deal with the $k=0$ case, as the $k=1$ case follows from a similar way.

For $|x| \leq M^{-1}$, we have a trivial bound that
\begin{align*}
\sum_{i,j}K_{i,j}(x)\lesssim M^{2}.
\end{align*}

For $|x|\geq M^{-1}$, we observe that
\begin{align*}
\sum_{i,j}K_{i,j}(x)=&\sum_{i,j}\chi (MP_{e_{i,j}}^{\perp }x)\chi (P_{e_{i,j}}x)\\
\leq&\chi(x)\sharp\lr{1\leq i\leq j \leq M:|P_{e_{i,j}}^{\perp}\frac{x}{|x|}|\leq \frac{2}{M|x|}}
\end{align*}
By the separation estimate \eqref{equ:gap estimate,eij}, the unit vectors $\lr{e_{i,j}}$ are mutually separated by a distance of at least $\frac{1}{M}$. This separation property implies that the spherical caps of diameter $\frac{1}{M}$ centered at these unit vectors are pairwise disjoint. A standard packing argument shows that the number of such disjoint caps is bounded above by the ratio of the sphere's total surface area to the minimal area occupied by an individual cap. Consequently, we have
\begin{align*}
\sharp\lr{1\leq i\leq j \leq M:|P_{e_{i,j}}^{\perp}\frac{x}{|x|}|\leq \frac{2}{M|x|}}\lesssim \lrs{\frac{2}{M|x|}}^{2}\div \lrs{\frac{1}{M}}^{2}\lesssim \frac{1}{|x|^{2}}.
\end{align*}

Combining the two cases, we obtain
\begin{align*}
\sum_{i,j}K_{i,j}(x)\lesssim& M^{2}\lrs{\frac{1}{M|x|+1}}^{2}\chi(x),
\end{align*}
which completes the proof of \eqref{equ:Kij,sum}.

\textbf{4. Proof of the disjointness estimate \eqref{equ:Iij,sum}.}

It suffices to show that the supports of $I_{i,j}$ are mutually disjoint. Specifically, we aim to prove that
\begin{align}\label{equ:disjoint support,Iij}
\operatorname{Supp}I_{i_{1},j_{1}}\bigcap \operatorname{Supp}I_{i_{2},j_{2}}=\emptyset,\quad
\text{for $(i_{1},j_{1})\neq (i_{2},j_{2})$}.
\end{align}
Let $v\in \operatorname{Supp}I_{i_{1},j_{1}}$ and $u\in \operatorname{Supp}I_{i_{2},j_{2}}$. Decomposing these vectors into their parallel and perpendicular components with respect to their respective directions gives
\begin{align}
&v=P_{e_{i_{1},j_{1}}}v+P_{e_{i_{1},j_{1}}}^{\perp}v,\quad |P_{e_{i_{1},j_{1}}}v|\in [\frac{9}{10},\frac{11}{10}],\quad |P_{e_{i_{1},j_{1}}}^{\perp}v|\leq \frac{1}{10M},\label{equ:v,Iij}\\
&u=P_{e_{i_{2},j_{2}}}u+P_{e_{i_{2},j_{2}}}^{\perp}u,\quad |P_{e_{i_{2},j_{2}}}u|\in [\frac{9}{10},\frac{11}{10}],\quad  |P_{e_{i_{2},j_{2}}}^{\perp}u|\leq \frac{1}{10M},\label{equ:u,Iij}
\end{align}

Using the algebraic identity
\begin{align*}
|ae_{i_{1},j_{1}}-be_{i_{2},j_{2}}|^{2}=&a^{2}+b^{2}-2ab\lra{e_{i_{1},j_{1}},e_{i_{2},j_{2}}}\\
=&(a-b)^{2}+2ab(1-\lra{e_{i_{1},j_{1}},e_{i_{2},j_{2}}})\notag\\
\geq&ab|e_{i_{1},j_{1}}-e_{i_{2},j_{2}}|^{2},\notag
\end{align*}
alongside the separation estimate \eqref{equ:gap estimate,eij} for $e_{i,j}$ and the magnitude bounds from \eqref{equ:v,Iij}--\eqref{equ:u,Iij}, we derive a lower bound for the relative velocity
\begin{align}\label{equ:lower bound,v,u,eij}
|v-u|\geq& |P_{e_{i_{1},j_{1}}}v-P_{e_{i_{2},j_{2}}}u|
-|P_{e_{i_{1},j_{1}}}^{\perp}v|
-|P_{e_{i_{2},j_{2}}}^{\perp}u|\\
\geq& \frac{|e_{i_{1},j_{1}}-e_{i_{2},j_{2}}|}{2}-|P_{e_{i_{1},j_{1}}}^{\perp}v|
-|P_{e_{i_{2},j_{2}}}^{\perp}u|\notag\\
\geq&\frac{|e_{i_{1},j_{1}}-e_{i_{2},j_{2}}|}{4}\geq \frac{1}{4M}.\notag
\end{align}

The estimate \eqref{equ:lower bound,v,u,eij} implies that the distance between any two points belonging to distinct supports is strictly positive, which establishes the disjointness property \eqref{equ:disjoint support,Iij}. Consequently, for any
given $v$, at most one term $I_{i_{0},j_{0}}(v)$ in the sum $\sum_{i,j}I_{i,j}(v)$ is non-zero.
 Furthermore, each term is supported within the annulus, that is,
\begin{align*}
\operatorname{Supp}I_{i,j}\subset \lr{v\in \R^{3}:\frac{3}{4}\leq|v|\leq \frac{5}{4}}
\end{align*}
Thus, we deduce
\begin{align*}
\sum_{i,j}I_{i,j}(v)\leq 1_{\lr{\frac{3}{4}\leq|v|\leq \frac{5}{4}}}(v),
\end{align*}
which concludes the proof.
\end{proof}
\subsection{Sobolev Norm Estimates for Approximate Solutions}\label{section:Sobolev Norm Estimates for Approximate Solutions}
We now establish Sobolev norm estimates for the components $f_{\mathrm{b}}$, $f_{\mathrm{r}}$, and the full approximation solution $f_{\mathrm{a}}$. Through these estimates, we deduce the norm deflation property for the approximate solution $f_{\mathrm{a}}$.

\begin{lemma}[Sobolev norm estimates for $f_{\mathrm{b}}$]\label{lemma:Hs,bounds on fb}
For $t\in [0,\frac{1}{4}]$, we have
\begin{align}
\n{e^{\lra{v}^{2}}f_{\mathrm{b}}(t)}_{L_{v}^{2}H_{x}^{s_{0}}}
\lesssim & M^{s_{0}-s},\label{equ:critical norm estimate for fb}\\
\n{e^{\lra{v}^{2}}f_{\mathrm{b}}(t)}_{L_{v}^{p}L_{x}^{\infty}}
+M^{-1}\n{e^{\lra{v}^{2}}\nabla_{x}f_{\mathrm{b}}(t)}_{L_{v}^{p}L_{x}^{\infty}}\lesssim& M^{1-s},\quad p\in[1,\infty],
\label{equ:sobolev norm,fb,LvpLxinf}\\
\n{e^{\lra{v}^{2}}f_{\mathrm{b}}(t)}_{L_{x}^{2}L_{v}^{1}}+
M^{-1}\n{e^{\lra{v}^{2}}\nabla_{x}f_{\mathrm{b}}(t)}_{L_{x}^{2}L_{v}^{1}}\lesssim& M^{-\frac{1}{2}-s}.\label{equ:sobolev norm,fb,L2L1}
\end{align}

\end{lemma}
\begin{proof}
Recall that
\begin{equation*}
f_{\mathrm{b}}(t,x,v)=M^{1-s}\sum_{i,j}K_{i,j}(x-vt)I_{i,j}(v).
\end{equation*}
Due to the uniform compact support of $f_{\mathrm{b}}$, the Gaussian weight $e^{\lra{v}^{2}}$ is uniformly bounded and can thus be omitted.
By the Minkowski inequality, we have
\begin{align*}
\n{\lra{\nabla_{x}}^{s_{0}}f_{\mathrm{b}}(t,x,v)}_{L_{v}^{2}L_{x}^{2}}\leq &
M^{1-s}\bbn{\sum_{i,j}\lra{\nabla_{x}}^{s_{0}}K_{i,j}(x-vt)I_{i,j}(v)}_{L_{v}^{2}L_{x}^{2}}\\
\leq&M^{1-s}\bbn{\sum_{i,j}\n{\lra{\nabla_{x}}^{s_{0}}K_{i,j}(x-vt)}_{L_{x}^{2}}I_{i,j}(v)}_{L_{v}^{2}}\notag\\
\leq& M^{1-s}M^{s_{0}}M^{-1}\bbn{\sum_{i,j}I_{i,j}(v)}_{L_{v}^{2}}\lesssim M^{s_{0}-s},\notag
\end{align*}
where in the last inequality we have employed the bounds
\begin{align*}
\n{\lra{\nabla_{x}}^{s_{0}}K_{i,j}(x)}_{L_{x}^{2}}\lesssim& M^{s_{0}}M^{-1},
\end{align*}
and the summation estimate \eqref{equ:Iij,sum} for $I_{i,j}$
\begin{align}
\bbn{\sum_{i,j}I_{i,j}(v)}_{L_{v}^{2}}\lesssim& \n{1_{\lr{\frac{3}{4}\leq|v|\leq \frac{5}{4}}}(v)}_{L_{v}^{2}}\lesssim 1.\label{equ:upper bound,Ij,sum}
\end{align}
Hence, we complete the proof of \eqref{equ:critical norm estimate for fb}.

For \eqref{equ:sobolev norm,fb,LvpLxinf}, using estimate \eqref{equ:upper bound,Ij,sum} again, we have
\begin{align*}
&\n{f_{\mathrm{b}}(t)}_{L_{v}^{p}L_{x}^{\infty}}
+M^{-1}\n{\nabla_{x}f_{\mathrm{b}}(t)}_{L_{v}^{p}L_{x}^{\infty}}\\
\leq &M^{1-s}\bbn{\sum_{i,j}K_{i,j}(x-vt)I_{i,j}(v)}_{L_{v}^{p}L_{x}^{\infty}}
+M^{-s}\bbn{\sum_{i,j}\nabla_{x}K_{i,j}(x-vt)I_{i,j}(v)}_{L_{v}^{p}L_{x}^{\infty}}\\
\lesssim& \lrs{M^{1-s}\sup_{j}\n{K_{i,j}}_{L_{x}^{\infty}}+M^{-s}\sup_{i,j}\n{\nabla_{x}K_{ij}}_{L_{x}^{\infty}}}\bbn{\sum_{i,j}I_{i,j}(v)}_{L_{v}^{p}}\\
\lesssim& M^{1-s}.
\end{align*}

For \eqref{equ:sobolev norm,fb,L2L1},
under the given constraints on the $v$-variable and for $t\in[0,\frac{1}{4}]$, we have the pointwise estimate
\begin{align}\label{equ:fb,pointwise,estimate,upper}
f_{\mathrm{b}}(t,x,v)=M^{1-s}\sum_{i,j}K_{i,j}(x-vt)I_{i,j}(v)\leq M^{1-s}\sum_{i,j}\wt{K}_{i,j}(x)I_{i,j}(v),
\end{align}
where
\begin{align*}
\wt{K}_{i,j}(x)=\chi(\frac{MP_{e_{i,j}}^{\perp }x}{10})\chi (\frac{P_{e_{i,j}}x}{10}).
\end{align*}
Following an argument similar to that for \eqref{equ:Kij,sum}, we have the following estimate
\begin{align}\label{equ:upper bound,Kj,sum}
\sum_{i,j}|\nabla_{x}^{k}\wt{K}_{i,j}(x)|\lesssim M^{2+k}\lrs{\frac{1}{M|x|+1}}^{2}\chi(\frac{x}{10}),\quad k=0,1.
\end{align}
Using this estimate \eqref{equ:upper bound,Kj,sum}, we derive
\begin{align*}
&\n{f_{\mathrm{b}}(t)}_{L_{x}^{2}L_{v}^{1}}+M^{-1}\n{\nabla_{x}f_{\mathrm{b}}(t)}_{L_{x}^{2}L_{v}^{1}}\\
\lesssim& M^{1-s}\bbn{\sum_{i,j}\wt{K}_{i,j}(x)\n{I_{i,j}}_{L_{v}^{1}}}_{L_{x}^{2}}+
M^{-s}\bbn{\sum_{i,j}\nabla_{x}\wt{K}_{i,j}(x)\n{I_{i,j}}_{L_{v}^{1}}}_{L_{x}^{2}}\\
\lesssim& M^{1-s}M^{-2}M^{2}\bbn{\lrs{\frac{1}{M|x|+1}}^{2}}_{L_{x}^{2}}\lesssim M^{-\frac{1}{2}-s},
\end{align*}
which completes the proof of \eqref{equ:sobolev norm,fb,L2L1}.
\end{proof}

Prior to analyzing $f_{\mathrm{r}}$, we first establish crucial pointwise estimates for $A[f_{\mathrm{b}}]$ given by \eqref{equ:Afb}.
\begin{lemma}[Pointwise estimates for $f_{\mathrm{b}}$]\label{lemma:pointwise estimate,fb}
Let $t\in[0,\frac{1}{4}]$.
For $k=0,1,2$, we have the pointwise upper bound
\begin{align}\label{equ:upper bound,fr,k}
\big|\chi(Nv) \nabla_{x}^{k}A[f_{\mathrm{b}}](t,x,v)\big |\lesssim M^{k+1-s}.
\end{align}
Moreover, we have the corresponding lower bound
\begin{equation}\label{equ:upper lower bound,fr}
\chi(Nv)A[f_{\mathrm{b}}] \gtrsim  M^{1-s} \chi(Nv),\quad \text{for $|x|\leq \frac{1}{100M}$}.
\end{equation}
\end{lemma}
\begin{proof}
Recall the expression \eqref{equ:Afb} that
\begin{align}\label{equ:Afb,pointwise}
A[f_{\mathrm{b}}](t,x,v)\sim \int_{\R^{3}}|u-v|f_{\mathrm{b}}(t,x,u)du.
\end{align}
By analyzing the supports of the variables $v$ and $u$, we observe that $|v|\sim N^{-1}$ and $|u|\sim 1$, which implies $|u-v|\sim 1$. This observation yields
\begin{align}\label{equ:pointwise estimate on fb,integral,final}
\chi(Nv)\int_{\R^{3}} |u-v|f_{\mathrm{b}
}(t,x,u)\,du
\sim&  \chi(Nv)\int_{\R^{3}} f_{\mathrm{b}
}(t,x,u)\,du.
\end{align}

For $0\leq t\leq \frac{1}{4}$, under the given constraints on the $x$-variable and $u$-variable, we have the pointwise estimate for $f_{\mathrm{b}}$
\begin{align}\label{equ:pointwise estimate on fb}
&M^{1-s}\sum_{i,j}\chi
	(10MP_{e_{i,j}}^{\perp }x)\chi (10P_{e_{i,j}}x)\chi
	(MP_{e_{i,j}}^{\perp }u)\chi (10P_{e_{i,j}}(u-e_{i,j})) \\
 \leq& f_{\mathrm{b}}(t,x,u)\leq M^{1-s}\sum_{i,j}\chi (\frac{
		MP_{e_{i,j}}^{\perp }x}{10})\chi (\frac{P_{e_{i,j}}x}{10})\chi
	(MP_{e_{i,j}}^{\perp }u)\chi (10P_{e_{i,j}}(u-e_{i,j})).\notag
\end{align}
Thus, for the upper bound \eqref{equ:upper bound,fr,k} with $k=0$, we use \eqref{equ:Afb,pointwise}--\eqref{equ:pointwise estimate on fb} to obtain
\begin{align*}
\chi(Nv) A[f_{\mathrm{b}}](t,x,v)\sim& \chi(Nv)\int_{\R^{3}}f_{\mathrm{b}}(t,x,u)du\\
\lesssim&M^{1-s}\chi(Nv)\sum_{i,j}\chi (\frac{MP_{e_{i,j}}^{\perp }x}{10})\chi (\frac{P_{e_{i,j}}x}{10})\n{I_{i,j}}_{L_{u}^{1}}\\
 \lesssim& M^{1-s}M^{-2}\chi(Nv)\sum_{i,j}\chi (\frac{MP_{e_{i,j}}^{\perp }x}{10})\chi (\frac{P_{e_{i,j}}x}{10})\lesssim M^{1-s}\chi(Nv).
\end{align*}
We note that the upper bounds of estimates \eqref{equ:pointwise estimate on fb,integral,final} and \eqref{equ:pointwise estimate on fb} remain true with the extra factor $M^{k}$ if $\nabla_x^k$
is applied, for $k\geq 0$. Therefore, we conclude the pointwise upper bound \eqref{equ:upper bound,fr,k} on $A[f_{\mathrm{b}}]$ for $k=0,1,2$.

For the lower bound \eqref{equ:upper lower bound,fr}, we observe that
\begin{align}\label{equ:lower bound,fb,sum}
\sum_{i,j}\chi (10MP_{e_{i,j}}^{\perp }x)\chi (10P_{e_{i,j}}x)\gtrsim M^{2},\quad \text{for $|x|\leq  \frac{1}{100M}$},
 \end{align}
Combining  \eqref{equ:pointwise estimate on fb,integral,final}, \eqref{equ:pointwise estimate on fb}, and \eqref{equ:lower bound,fb,sum}, we obtain
\begin{align*}
\chi(Nv) A[f_{\mathrm{b}}](t,x,v)\gtrsim & \chi(Nv)\int_{\R^{3}}f_{\mathrm{b}
}(t,x,u)du \\
\geq &\chi(Nv)M^{1-s} \sum_{i,j}\chi (10MP_{e_{i,j}}^{\perp }x)\chi (10P_{e_{i,j}}x)\n{I_{i,j}}_{L_{u}^{1}}\\
\gtrsim& M^{1-s} \chi(Nv),
\end{align*}
which completes the proof of \eqref{equ:upper lower bound,fr}.
\end{proof}

We now proceed to establish the Sobolev norm estimates for $f_{\mathrm{r}}$.
\begin{lemma}[Sobolev norm estimates for $f_{\mathrm{r}}$]\label{lemma:bounds on fr}
For $t\in[0,\frac{1}{4}]$, we have the upper bound
\begin{align}
&\n{e^{\lra{v}^{2}}f_{\mathrm{r}}(t)}_{L_{v}^{2}H_{x}^{s_{0}}}\lesssim  \exp[-t M^{1-s}],\label{equ:critical norm estimate for fr,upper bound}
\end{align}
and the lower bound
\begin{align}
&\n{f_{\mathrm{r}}(0)}_{L_{v}^{2}H_{x}^{s_{0}}}\gtrsim  1.\label{equ:critical norm estimate for fr,lower bound}
\end{align}
In particular, the following estimates hold
\begin{align}
\n{e^{\lra{v}^{2}}f_{\mathrm{r}}(0)}_{L_{v}^{2}H_{x}^{s_{0}}}\sim & \n{f_{\mathrm{r}}(0)}_{L_{v}^{2}H_{x}^{s_{0}}}\sim 1,\label{equ:upper bound,fr,t=0}\\
\n{e^{\lra{v}^{2}}f_{\mathrm{r}}(T_{*})}_{L_{v}^{2}H_{x}^{s_{0}}}\lesssim & \frac{1}{ \ln M},\label{equ:lower bound,fr,T}
\end{align}
where $T_{*}=M^{s-1}(\ln\ln M)$.
\end{lemma}

\begin{proof}
Recall that
\begin{equation*}
f_{\mathrm{r}}(t,x,v)=M^{\frac{3}{2}-s_{0}}N^{\frac{3}{2}}
\exp \lrc{-\int_{0}^{t}A[f_{\mathrm{b}}](\tau,x,v)d\tau} \chi (100Mx)\chi (
Nv).
\end{equation*}
Since $f_{\mathrm{r}}$ has uniform compact support in $v$, the Gaussian weight $e^{\lra{v}^{2}}$ is uniformly bounded and can therefore be omitted from the subsequent estimates.

To establish the upper bound \eqref{equ:critical norm estimate for fr,upper bound}, we first derive the estimate for $\n{\nabla_{x}f_{\mathrm{r}}}_{L_{v}^{2}L_{x}^{2}}$.
Applying the pointwise upper bound \eqref{equ:upper bound,fr,k} and the lower bound \eqref{equ:upper lower bound,fr} for $A[f_{\mathrm{b}}]$, we get
\begin{align}\label{equ:upper bound,fr,sobolev}
	&\n{e^{\lra{v}^{2}}\nabla_{x}f_{\rm{r}}(t)}_{L_{v}^{2}L_{x}^{2}}\\
	\leq&M^{1+\frac{3}{2}-s_{0}}N^{\frac{3}{2}}\bbn{\exp\lrc{
-\int_{0}^{t}A[f_{\mathrm{b}}](\tau,x,v)d\tau}\chi(100Mx)\chi(Nv)}_{L_{v}^{2}L_{x}^{2}}\notag\\ &+M^{\frac{3}{2}-s_{0}}N^{\frac{3}{2}}\bbn{\nabla_{x}A[f_{\mathrm{b}}](\tau,x,v)\exp\lrc{
-\int_{0}^{t}A[f_{\mathrm{b}}](\tau,x,v)d\tau}\chi(100Mx)\chi(Nv)}_{L_{v}^{2}L_{x}^{2}}\notag\\
	\lesssim&
	M^{1+\frac{3}{2}-s_{0}}N^{\frac{3}{2}}\exp[-2tM^{1-s}]\n{(\nabla\chi)(100Mx)}_{L_{x}^{2}}\n{\chi(Nv)}_{L_{v}^{2}}\notag\\
	&+M^{\frac{3}{2}-s_{0}}N^{\frac{3}{2}}\lra{tM^{1+1-s}}\exp[-2tM^{1-s}]\n{\chi(100Mx)}_{L_{x}^{2}}\n{\chi(Nv)}_{L_{v}^{2}}\notag\\
	\lesssim& M^{1-s_{0}}\exp[-t M^{1-s}].\notag
\end{align}
In the same manner, the $L^{2}$ norm is controlled by
\begin{align}\label{equ:upper bound,fr,sobolev,L2}
\n{e^{\lra{v}^{2}}f_{\mathrm{r}}(t)}_{L_{v}^{2} L_{x}^{2}}\lesssim M^{-s_{0}}\exp[-t M^{1-s}].
\end{align}
Applying the interpolation inequality between these two estimates \eqref{equ:upper bound,fr,sobolev}--\eqref{equ:upper bound,fr,sobolev,L2} yields
\begin{align*}
\n{e^{\lra{v}^{2}}f_{\mathrm{r}}(t)}_{L_{v}^{2} H_x^{s_{0}}}\leq& \n{e^{\lra{v}^{2}}f_{\mathrm{r}}(t)}_{L_{v}^{2} H_{x}^{1}}^{s_{0}}
\n{e^{\lra{v}^{2}}f_{\mathrm{r}}(t)}_{L_{v}^{2} L_{x}^{2}}^{1-s_{0}}
\lesssim
 \exp[-t M^{1-s}],
\end{align*}
which completes the proof of \eqref{equ:critical norm estimate for fr,upper bound}.

For the lower bound estimate \eqref{equ:critical norm estimate for fr,lower bound}, we employ the Sobolev inequality to get
\begin{align*}
\n{\lra{\nabla_{x}}^{s_{0}}f_{\mathrm{r}}(0,x,v)}_{L_{v}^{2}L_{x}^{2}}
\gtrsim&\n{f_{\mathrm{r}}(0,x,v)}_{L_{v}^{2}L_{x}^{\frac{6}{3-2s_{0}}}}\\
\gtrsim&M^{\frac{3}{2}-s_{0}}N^{\frac{3}{2}}\n{ \chi (100Mx)\chi (
Nv)}_{L_{v}^{2}L_{x}^{\frac{6}{3-2s_{0}}}}\\
\gtrsim& M^{\frac{3}{2}-s_{0}}N^{\frac{3}{2}} \n{\chi(100Mx)}_{L_{x}^{\frac{6}{3-2s_{0}}}}\n{\chi(Nv)}_{L_{v}^{2}}\\
\gtrsim&1 .
\end{align*}
Hence, we have done the proof of estimate $(\ref{equ:critical norm estimate for fr,lower bound})$.

Finally, setting $t=0$ and $T_{*}= M^{s-1}(\ln\ln M)$ in \eqref{equ:critical norm estimate for fr,upper bound}--\eqref{equ:critical norm estimate for fr,lower bound}, we arrive at the desired estimates \eqref{equ:upper bound,fr,t=0}--\eqref{equ:lower bound,fr,T}.
\end{proof}

Finally, we conclude the norm deflation property for the approximate solution $f_{\mathrm{a}}$.
\begin{proposition}[Norm deflation property for $f_{\mathrm{a}}$]\label{lemma:norm deflation of fa}
Let $T_{*}= M^{s-1}(\ln\ln M)$. Then, we have
\begin{align}
	\n{e^{\lra{v}^{2}}f_{\mathrm{a}}(0)}_{L_v^{2} H_x^{s_{0}}}\sim&  1,\label{equ:critical norm estimate for fa,t=0}\\
	\n{e^{\lra{v}^{2}}f_{\mathrm{a}}(T_{*})}_{L_v^{2} H_x^{s_{0}}}\lesssim& \frac{1}{\ln M}.\label{equ:critical norm estimate for fa,t=T}
\end{align}
\end{proposition}

\begin{proof}
Since $f_{\mathrm{b}}$ and $f_{\mathrm{r}}$ have disjoint velocity supports, the norm of their sum satisfies
\begin{align*}
\n{e^{\lra{v}^{2}}f_{\mathrm{a}}(t)}_{L_v^{2} H_x^{s_{0}}}^{2}\sim \n{e^{\lra{v}^{2}}f_{\mathrm{r}}(t)}_{L_v^{2} H_x^{s_{0}}}^{2}+\n{e^{\lra{v}^{2}}f_{\mathrm{b}}(t)}_{L_v^{2} H_x^{s_{0}}}^{2}.
\end{align*}
The estimate for $f_{\mathrm{b}}$ is provided by Lemma \ref{lemma:Hs,bounds on fb}, specifically
\begin{align*}
\sup_{t\in[0,T_{*}]}\n{e^{\lra{v}^{2}}f_{\mathrm{b}}(t,x,v)}_{L_{v}^{2}H_{x}^{s_{0}}}\lesssim M^{s_{0}-s}\lesssim \frac{1}{\ln M}.
\end{align*}
For $f_{\mathrm{r}}$, from Lemma \ref{lemma:bounds on fr} we have
\begin{align*}
\n{e^{\lra{v}^{2}}f_{\mathrm{r}}(0)}_{L_{v}^{2}H_{x}^{s_{0}}}\sim & \n{f_{\mathrm{r}}(0)}_{L_{v}^{2}H_{x}^{s_{0}}}\sim 1,\\
\n{e^{\lra{v}^{2}}f_{\mathrm{r}}(T_{*})}_{L_{v}^{2}H_{x}^{s_{0}}}\lesssim & \frac{1}{\ln M}.
\end{align*}
Combining these above bounds, we directly obtain \eqref{equ:critical norm estimate for fa,t=0} and \eqref{equ:critical norm estimate for fa,t=T}.
\end{proof}

\section{Space-time Collision Estimates in the Gaussian-weighted $Z$-Norm}\label{section:Space-time Collision}
A direct calculation shows that the correction term $f_{\mathrm{c}}$ satisfies
\begin{equation}\label{equ:correction term,fc}
\left\{
\begin{aligned} \partial_t f_{\mathrm{c}} + v\cdot \nabla_x f_{\mathrm{c}} = & \pm Q^\pm(f_{\mathrm{c}},f_{ \mathrm{a}}) \pm
Q^\pm(f_{\mathrm{a}},f_{\mathrm{c}}) \pm Q^\pm(f_{\mathrm{c}},f_{\mathrm{c}
})-F_{\text{err}},\\
F_{\mathrm{err}}=&\pa_{t}f_{\mathrm{a}}+v\cdot \nabla_{x}f_{\mathrm{a}}+Q^{-}(f_{\mathrm{a}},f_{\mathrm{a}})-Q^{+}(f_{\mathrm{a}},f_{\mathrm{a}}).
 \end{aligned}
 \right.
\end{equation}
To construct the exact solution as a perturbation of the approximate solution, one relies on the well-posedness theory for the equation \eqref{equ:correction term,fc}. More importantly, we must prove that $f_{\mathrm{c}}(t)$ can be extended up to the time $T_{*}$ when the ill-posedness behavior occurs.
Usually, to deal with this problem, one expects a time-independent bilinear estimate of the form
\begin{align}\label{equ:bilinear estimate,time-independent}
\n{Q^{\pm}(f,g)}_{Z}\lesssim \n{f}_{Z}\n{g}_{Z}.
\end{align}
Unfortunately, because of the loss of velocity weight inherent in the hard-sphere collision kernel, such an inequality like \eqref{equ:bilinear estimate,time-independent} generally fails to hold. To circumvent this obstacle, we instead establish Gaussian-weighted $Z$-norm collision estimates in time-integral form.

The choice of $Z$-norm requires a delicate balance. On one hand, the norm must be sufficiently strong to close the bilinear space-time estimates to control the Sobolev norm $L_{v}^{2}H_{x}^{s_{0}}$. On the other hand, it needs to be as weak as possible to ensure that the
$Z$-norm of the error terms remains uniformly controllable in parameter $M$. This is crucial because the higher regularity norms of both $f_{\mathrm{a}}^{M}$ and the error term $F_{\mathrm{err}}^{M}$ exhibit rapid growth in $M$.
Here, the $Z$-norm we construct is given by
\begin{align}\label{equ:Z-norm}
\n{f}_{Z}:=M^{s_{0}}\n{f}_{L_{v}^{2,1}L_{x}^{2}}+M^{s_{0}-1}\n{\nabla_{x}f}_{L_{v}^{2,1}L_{x}^{2}}+\n{f}_{L_{v}^{1,1}L_{x}^{\infty}}+
M^{-1}\n{\nabla_{x}f}_{L_{v}^{1,1}L_{x}^{\infty}},
\end{align}
and $\n{f}_{L_{v}^{p,1}}=\n{\lra{v}f}_{L_{v}^{p}}$.

The main estimates in this section are summarized in Proposition \ref{lemma:Z,norm,bilinear,estimate}.
While space-time estimates of the form \eqref{equ:Z,norm,bilinear,estimate} are sufficient to establish local well-posedness, our final goal is to demonstrate that the correction term $f_{\mathrm{c}}(t)$ persists up to the moment when ill-posed behavior emerges. In the context, the time integration is employed to compensate for the loss of velocity weight. Consequently, relying solely on the smallness of the remaining time interval $(t-t_{0})^{\frac{1}{2}}$  is insufficient for such a purpose. To resolve this and ensure the stability analysis of the iterative framework in Section \ref{section:Stability Analysis for the Approximate Solutions via an Iterative Scheme}, we introduce a rate parameter
$\kappa$ to provide further smallness.

\begin{proposition}\label{lemma:Z,norm,bilinear,estimate}
 Let $0\leq t_{0}\leq t\leq \frac{1}{\kappa}$. Then the following time-dependent estimate in time-integral form holds
\begin{align}\label{equ:Z,norm,bilinear,estimate}
&\bbn{e^{(1-\kappa t)\lra{v}^{2}}\int_{t_{0}}^{t}e^{-(t-\tau)v\cdot \nabla_{x}}Q^{\pm}(f,g)(\tau)d\tau}_{Z}\\
\lesssim & \lrs{\frac{t-t_{0}}{\kappa}}^{\frac{1}{2}}\n{e^{(1-\kappa \tau)\lra{v}^{2}}f}_{L_{\tau}^{\infty}(t_{0},t;Z)}
\n{e^{(1-\kappa \tau)\lra{v}^{2}}g}_{L_{\tau}^{\infty}(t_{0},t;Z)}.\notag
\end{align}
For the time-independent case, with a loss of velocity weight, we have
\begin{align}
\n{e^{\lra{v}^{2}}Q^{\pm}(f,g)}_{Z}\lesssim \n{\lra{v}e^{\lra{v}^{2}}f}_{Z}
\lrs{\n{\lra{v}e^{\lra{v}^{2}}g}_{L_{v}^{1,1}L_{x}^{\infty}}
+M^{-1}\n{\lra{v}e^{\lra{v}^{2}}\nabla_{x}g}_{L_{v}^{1,1}L_{x}^{\infty}}},\label{equ:Z,norm,bilinear,estimate,time-independent}\\
\n{e^{\lra{v}^{2}}Q^{+}(f,g)}_{Z}\lesssim
\lrs{\n{\lra{v}e^{\lra{v}^{2}}f}_{L_{v}^{1,1}L_{x}^{\infty}}
+M^{-1}\n{\lra{v}e^{\lra{v}^{2}}\nabla_{x}f}_{L_{v}^{1,1}L_{x}^{\infty}}}\n{\lra{v}e^{\lra{v}^{2}}g}_{Z}.
\label{equ:Z,norm,bilinear,estimate,time-independent,gZ}
\end{align}
\end{proposition}
\begin{proof}
For notational brevity, we denote
\begin{align*}
N^{\pm}[f,g](t)=\int_{t_{0}}^{t}e^{-(t-\tau)v\cdot \nabla_{x}}Q^{\pm}(f,g)(\tau)d\tau.
\end{align*}
Applying H\"{o}lder's inequality in time and using the non-negativity of the collision operator, we obtain the pointwise estimate
\begin{align}\label{equ:space-time,pointwise,N}
&\bbn{e^{(1-\kappa t)\lra{v}^{2}}
N^{\pm}[f,g](t)}_{L_{x}^{p}}\\
\leq &
\int_{t_{0}}^{t}e^{(1-\kappa t)\lra{v}^{2}}\n{Q^{\pm}(f,g)(\tau)}_{L_{x}^{p}}d\tau\notag\\
\leq &\int_{t_{0}}^{t}e^{(1-\kappa t)\lra{v}^{2}}Q^{\pm}(\n{f}_{L_{x}^{p}},\n{g}_{L_{x}^{\infty}})d\tau\notag\\
= &\int_{t_{0}}^{t}e^{\kappa (\tau-t)\lra{v}^{2}}e^{(1-\kappa \tau)\lra{v}^{2}}Q^{\pm}(\n{f}_{L_{x}^{p}},\n{g}_{L_{x}^{\infty}})(\tau)d\tau\notag\\
\lesssim &
\int_{t_{0}}^{t}\frac{1}{\sqrt{\ka(t-\tau)}\lra{v}}Q^{\pm}\lrs{e^{(1-\kappa \tau)\lra{v}^{2}}\n{f}_{L_{x}^{p}},e^{(1-\kappa \tau)\lra{v}^{2}}\n{g}_{L_{x}^{\infty}}}(\tau)d\tau,\notag
\end{align}
where in the last inequality we have used that $e^{-a\lra{v}^{2}}\sqrt{a}\lra{v}\lesssim 1 $ and the energy conservation law which yields the inequality that
$e^{\lra{v}^{2}}\lesssim e^{\lra{v^{*}}^{2}+\lra{u^{*}}^{2}}$.

We now turn to the estimates in the $Z$-norm.

\textbf{$M^{s_{0}}\n{\cdot}_{L_{v}^{2}L_{x}^{2}}$ norm estimate}.

Taking $p=2$ in \eqref{equ:space-time,pointwise,N} and employing the collision estimates \eqref{equ:collision,estimate,Q-,Lr}--\eqref{equ:collision,estimate,Q+,Lr} with \(r=2\), we obtain
\begin{align}\label{equ:Z,norm,bilinear,L2}
&M^{s_{0}}\bbn{e^{(1-\kappa t)\lra{v}^{2}}
N^{\pm}[f,g](t)}_{L_{v}^{2,1}L_{x}^{2}}\\
\lesssim& M^{s_{0}}\int_{t_{0}}^{t}\frac{1}{\sqrt{\ka(t-\tau)}}
\bbn{Q^{\pm}\lrs{e^{(1-\kappa \tau)\lra{v}^{2}}\n{f}_{L_{x}^{2}},e^{(1-\kappa \tau)\lra{v}^{2}}\n{g}_{L_{x}^{\infty}}}(\tau)}_{L_{v}^{2}}d\tau\notag\\
\lesssim& M^{s_{0}}\int_{t_{0}}^{t}\frac{1}{\sqrt{\ka(t-\tau)}}
\bbn{e^{(1-\ka \tau)\lra{v}^{2}}\n{f}_{L_{x}^{2}}}_{L_{v}^{2,1}}
\bbn{e^{(1-\ka \tau)\lra{v}^{2}}\n{g}_{L_{x}^{\infty}}}_{L_{v}^{1,1}}d\tau\notag\\
\lesssim& \lrs{\frac{t-t_{0}}{\kappa}}^{\frac{1}{2}}
M^{s_{0}}\n{e^{(1-\kappa \tau)\lra{v}^{2}}f}_{L_{\tau}^{\infty}(t_{0},t;L_{v}^{2,1}L_{x}^{2})}
\n{e^{(1-\kappa \tau)\lra{v}^{2}}g}_{L_{\tau}^{\infty}(t_{0},t;L_{v}^{1,1}L_{x}^{\infty})}.\notag
\end{align}

\textbf{$M^{s_{0}-1}\n{\nabla_{x}\cdot}_{L_{v}^{2,1}L_{x}^{2}}$ norm estimate}.

Applying the Leibniz rule and estimate \eqref{equ:Z,norm,bilinear,L2}, we have
\begin{align}
&M^{s_{0}-1}\bbn{e^{(1-\kappa t)\lra{v}^{2}}\nabla_{x}
N^{\pm}[f,g](t)}_{L_{v}^{2,1}L_{x}^{2}}\\
=&M^{s_{0}-1}\bbn{e^{(1-\kappa t)\lra{v}^{2}}
N^{\pm}[\nabla_{x}f,g](t)}_{L_{v}^{2,1}L_{x}^{2}}+M^{s_{0}-1}\bbn{e^{(1-\kappa t)\lra{v}^{2}}
N^{\pm}[f,\nabla_{x}g](t)}_{L_{v}^{2,1}L_{x}^{2}}\notag\\
\lesssim& \lrs{\frac{t-t_{0}}{\kappa}}^{\frac{1}{2}}M^{s_{0}-1}
\n{e^{(1-\kappa \tau)\lra{v}^{2}}\nabla_{x}f}_{L_{\tau}^{\infty}(t_{0},t;L_{v}^{2,1}L_{x}^{2})}
\n{e^{(1-\kappa \tau)\lra{v}^{2}}g}_{L_{\tau}^{\infty}(t_{0},t;L_{v}^{1,1}L_{x}^{\infty})}\notag\\
&+\lrs{\frac{t-t_{0}}{\kappa}}^{\frac{1}{2}}
M^{s_{0}}\n{e^{(1-\kappa \tau)\lra{v}^{2}}f}_{L_{\tau}^{\infty}(t_{0},t;L_{v}^{2,1}L_{x}^{2})}
M^{-1}\n{e^{(1-\kappa \tau)\lra{v}^{2}}\nabla_{x}g}_{L_{\tau}^{\infty}(t_{0},t;L_{v}^{1,1}L_{x}^{\infty})}.\notag
\end{align}

\textbf{$\n{\cdot}_{L_{v}^{1,1}L_{x}^{\infty}}$ norm estimate}.

Using estimate \eqref{equ:space-time,pointwise,N} with $p=\infty$, and collision estimates \eqref{equ:collision,estimate,Q-,Lr}--\eqref{equ:collision,estimate,Q+,Lr} with $r=1$, we have
\begin{align}\label{equ:Z,norm,bilinear,L1}
&\bbn{e^{(1-\kappa t)\lra{v}^{2}}
N^{\pm}[f,g](t)}_{L_{v}^{1,1}L_{x}^{\infty}}\\
\lesssim& \int_{t_{0}}^{t}\frac{1}{\sqrt{\ka(t-\tau)}}
\bbn{Q^{\pm}\lrs{e^{(1-\kappa \tau)\lra{v}^{2}}\n{f}_{L_{x}^{\infty}},e^{(1-\kappa \tau)\lra{v}^{2}}\n{g}_{L_{x}^{\infty}}}(\tau)}_{L_{v}^{1}}d\tau \notag\\
\lesssim& \int_{t_{0}}^{t}\frac{1}{\sqrt{\ka(t-\tau)}}
\bbn{e^{(1-\ka \tau)\lra{v}^{2}}\n{f}_{L_{x}^{\infty}}}_{L_{v}^{1,1}}
\bbn{e^{(1-\ka \tau)\lra{v}^{2}}\n{g}_{L_{x}^{\infty}}}_{L_{v}^{1,1}}d\tau\notag\\
\lesssim& \lrs{\frac{t-t_{0}}{\kappa}}^{\frac{1}{2}}\n{e^{(1-\kappa \tau)\lra{v}^{2}}f}_{L_{\tau}^{\infty}(t_{0},t;L_{v}^{1,1}L_{x}^{\infty})}\n{e^{(1-\kappa \tau)\lra{v}^{2}}g}_{L_{\tau}^{\infty}(t_{0},t;L_{v}^{1,1}L_{x}^{\infty})}.\notag
\end{align}

\textbf{$M^{-1}\n{\nabla_{x}\cdot}_{L_{v}^{1,1}L_{x}^{\infty}}$ norm estimate}.

Finally, applying the Leibniz rule and using \eqref{equ:Z,norm,bilinear,L1}, we derive
\begin{align}
&M^{-1}\bbn{e^{(1-\kappa t)\lra{v}^{2}}\nabla_{x}
N^{\pm}[f,g](t)}_{L_{v}^{1,1}L_{x}^{\infty}}\\
&M^{-1}\bbn{e^{(1-\kappa t)\lra{v}^{2}}
N^{\pm}[\nabla_{x}f,g](t)}_{L_{v}^{1,1}L_{x}^{\infty}}+M^{-1}\bbn{e^{(1-\kappa t)\lra{v}^{2}}
N^{\pm}[f,\nabla_{x}g](t)}_{L_{v}^{1,1}L_{x}^{\infty}}\notag\\
\lesssim& \lrs{\frac{t-t_{0}}{\kappa}}^{\frac{1}{2}}M^{-1}\n{e^{(1-\kappa \tau)\lra{v}^{2}}\nabla_{x}f}_{L_{\tau}^{\infty}(t_{0},t;L_{v}^{1,1}L_{x}^{\infty})}\n{e^{(1-\kappa \tau)\lra{v}^{2}}g}_{L_{\tau}^{\infty}(t_{0},t;L_{v}^{1,1}L_{x}^{\infty})}\notag\\
&+\lrs{\frac{t-t_{0}}{\kappa}}^{\frac{1}{2}}\n{e^{(1-\kappa \tau)\lra{v}^{2}}f}_{L_{\tau}^{\infty}(t_{0},t;L_{v}^{1,1}L_{x}^{\infty})}M^{-1}\n{e^{(1-\kappa \tau)\lra{v}^{2}}\nabla_{x}g}_{L_{\tau}^{\infty}(t_{0},t;L_{v}^{1,1}L_{x}^{\infty})}.\notag
\end{align}
Combining the four estimates above, we arrive at the desired estimate \eqref{equ:Z,norm,bilinear,estimate}.

The time-independent estimates \eqref{equ:Z,norm,bilinear,estimate,time-independent}--\eqref{equ:Z,norm,bilinear,estimate,time-independent,gZ} are obtained by a similar argument to that for \eqref{equ:Z,norm,bilinear,estimate}, with an extra factor
$\lra{v}$ introduced to compensate for the loss of velocity weight. Thus, we have completed the proof.

\end{proof}

\section{$Z$-norm Estimates for Error Terms}\label{section:Estimates for Error Terms}
This section is devoted to deriving the $Z$-norm estimates for the error term $F_{\mathrm{err}}$. We begin by recalling the error terms
\begin{align}\label{equ:Ferr,error term estimate}
F_{\text{err}}& =\partial _{t}f_{\mathrm{a}}+v\cdot \nabla
_{x}f_{\mathrm{a}}+Q^{-}(f_{\mathrm{a}},f_{\mathrm{a}})-Q^{+}(f_{\mathrm{a}},f_{
\mathrm{a}}) \\ & =v\cdot \nabla
_{x}\,f_{\mathrm{r}}-Q^{+}(f_{\mathrm{r}},f_{\mathrm{b}})\mp Q^{\pm}(f_{\mathrm{b}},f_{\mathrm{r}})\mp Q^{\pm}(f_{\mathrm{r}},f_{
\mathrm{r}})\mp Q^{\pm}(f_{\mathrm{b}},f_{\mathrm{b}}),\notag
\end{align}
and the $Z$-norm
\begin{align}\label{equ:Z-norm,error term estimate}
\n{f}_{Z}=M^{s_{0}}\n{f}_{L_{v}^{2,1}L_{x}^{2}}+M^{s_{0}-1}\n{\nabla_{x}f}_{L_{v}^{2,1}L_{x}^{2}}
+\n{f}_{L_{v}^{1,1}L_{x}^{\infty}}+
M^{-1}\n{\nabla_{x}f}_{L_{v}^{1,1}L_{x}^{\infty}}.
\end{align}

The main result of this section is summarized in the following proposition.
\begin{proposition}[$Z$-norm estimates for $F_{\text{err}}$]\label{lemma:bounds on ferr}
 For $0\leq  t_{0}\leq t \leq T_{*}$, we have
\begin{equation}\label{equ:Ferr_bound}
\bbn{e^{\lra{v}^{2}}\int_{t_{0}}^{t}e^{-(t-\tau)v\cdot \nabla _{x}}F_{\text{err}
}(\tau)\,d\tau}_{Z}\lesssim M^{-\delta},
\end{equation}
for some $\delta>0$.
\end{proposition}

Given that the
 $Z$-norm \eqref{equ:Z-norm,error term estimate} comprises four sub-norms and there are eight error terms in \eqref{equ:Ferr,error term estimate}, a total of 32 error estimates must be established. For the majority of the terms,
our primary strategy is to move the time integration outside the norm, leading to an estimate of the form
\begin{equation*}
	\bbn{e^{\lra{v}^{2}} \int_{t_{0}}^{t}e^{-(t-\tau)v\cdot \nabla _{x}}F_{\text{err}%
	}(\tau)\,d\tau}_{Z}\lesssim T_{*}\n{e^{\lra{v}^{2}}F_{\text{err}}}_{L_{t}^{\infty }Z}.
\end{equation*}
However, a notable exception occurs in the treatment of the  $L_{v}^{1,1}L_{x}^{\infty
}$ norm of $Q^{\pm}(f_{\mathrm{b}},f_{\mathrm{b}})$, where extracting the time integral yields a rapid growth in $M$.
To circumvent this, we employ time-integration techniques from dispersive equations. This delicate analysis requires a variety of analytical techniques, including localization via time integration, projection measures, angular decomposition, lattice counting, and geometric estimates.

The overall proof strategy relies on establishing suitable
 $Z$-norm bounds for the approximate solutions as a prerequisite. Accordingly, in Section \ref{sec:Z-norm Bounds on the Approximation Profiles}, we first establish such $Z$-norm bounds for $f_{\mathrm{b}}$, $f_{\mathrm{r}}$, and $f_{\mathrm{a}}$. Building on these bounds, we analyze each term appearing in $F_{\mathrm{err}}$ separately.
Specifically, in Section \ref{section:Analysis of term1} we deal with the estimate of the transport term $v \cdot \nabla_{x}f_{\mathrm{r}}$, while in Section \ref{section:Analysis of term2} we address the bilinear terms involving $f_{\mathrm{r}}$. Finally, the most delicate terms $Q^{\pm}(f_{\mathrm{b}},f_{\mathrm{b}})$ are handled in Sections \ref{section:Analysis of Q-} and \ref{section:Analysis of Q+}.

\subsection{$Z$-norm Estimates for the Approximation Solutions}\label{sec:Z-norm Bounds on the Approximation Profiles}
In the section, we focus on establishing the $Z$-norm estimates for $f_{\mathrm{b}}$ and $f_{\mathrm{r}}$ and $f_{\mathrm{a}}$.
\begin{lemma}[$Z$-norm estimates for $f_{\mathrm{b}}$, $f_{\mathrm{r}}$, $f_{\mathrm{a}}$]\label{lemma:z-norm bounds on,fb,fr,fa}
For the $Z$-norm, we have
\begin{align}
\n{e^{\lra{v}^{2}}f_{\mathrm{b}}(t)}_{L^{\infty}(0,T_{*};Z)}\lesssim&  M^{1-s},\label{equ:z-norm estimate for fb}\\
\n{e^{\lra{v}^{2}}f_{\mathrm{r}}(t)}_{L_{t}^{\infty}(0,T_{*};Z)}\lesssim& 1,\label{equ:z-norm bounds for fr}\\
\n{e^{\lra{v}^{2}}f_{\mathrm{a}}(t)}_{L_{t}^{\infty}(0,T_{*};Z)}\lesssim& M^{1-s}.\label{equ:z-norm estimate for fa}
\end{align}
\end{lemma}
\begin{proof}
We begin by establishing the $Z$-norm estimate for $f_{\mathrm{b}}$.
Recalling the Sobolev norm estimates in Lemma \ref{lemma:Hs,bounds on fb}, we have
\begin{align}
\n{e^{\lra{v}^{2}}f_{\mathrm{b}}(t)}_{L_{v}^{2}H_{x}^{s_{0}}}
\lesssim & M^{s_{0}-s},\label{equ:critical norm estimate for fb,used}\\
\n{e^{\lra{v}^{2}}f_{\mathrm{b}}(t)}_{L_{v}^{p}L_{x}^{\infty}}\lesssim& M^{1-s},\quad p\in[1,\infty].
\label{equ:sobolev norm,fb,LvpLxinf,used}
\end{align}
Using the uniform compact support of $f_{\mathrm{b}}$ together with the above estimates, we obtain
\begin{align}\label{equ:fb,Z-norm,Lv2Lx2}
M^{s_{0}} \n{e^{\lra{v}^{2}}f_{\mathrm{b}}}_{L_{v}^{2,1} L_{x}^{2}} + \n{e^{\lra{v}^{2}}f_{\mathrm{b}}}_{L_{v}^{1,1} L_{x}^{\infty}} \lesssim M^{1-s}.
\end{align}
Furthermore, we address the estimates involving the spatial derivative.
Noting that the $x$-derivative $\nabla_{x}$ acting on $f_{\mathrm{b}}$ yields an additional factor $M$, we apply the same argument used to derive \eqref{equ:fb,Z-norm,Lv2Lx2} to yield
\begin{align}\label{equ:fb,Z-norm,H1,Lv2Lx2}
M^{s_{0}-1} \n{e^{\lra{v}^{2}}\nabla_{x} f_{\mathrm{b}}}_{L_{v}^{2,1} L_{x}^{2}} + M^{-1} \n{e^{\lra{v}^{2}}\nabla_{x} f_{\mathrm{b}}}_{L_{v}^{1,1} L_{x}^{\infty}} \lesssim M^{1-s}.
\end{align}
Combining \eqref{equ:fb,Z-norm,Lv2Lx2} and \eqref{equ:fb,Z-norm,H1,Lv2Lx2} gives the desired $Z$-norm estimate \eqref{equ:z-norm estimate for fb} for $f_{\mathrm{b}}$.

Next, we turn to the analysis of the $Z$-norm estimate for $f_{\mathrm{r}}$.
Let us recall that
\begin{equation*}
f_{\mathrm{r}}(t,x,v)=M^{\frac{3}{2}-s_{0}}N^{\frac{3}{2}}\exp \left[
-\int_{0}^{t}A[f_{\mathrm{b}}](\tau,x,v)d\tau\right] \chi (100Mx)\chi (Nv).
\end{equation*}
We proceed to carry out the
$Z$-norm estimates.

\textbf{The $M^{s_{0}}\Vert e^{\lra{v}^{2}} \cdot \Vert_{L_{v}^{2,1}L_{x}^{2}}$ and $M^{s_{0}-1}\Vert e^{\lra{v}^{2}}\nabla_{x}\cdot\Vert_{L_{v}^{2,1}L_{x}^{2}}$
 estimates.}

Using the uniform compact support of $f_{\mathrm{r}}$ together with the Sobolev norm estimates \eqref{equ:upper bound,fr,sobolev}--\eqref{equ:upper bound,fr,sobolev,L2} established for $f_{\mathrm{r}}$, we immediately get
 \begin{align*}
 M^{s_{0}} \n{e^{\lra{v}^{2}}f_{\rm{r}}(t)}_{L_{v}^{2,1}L_{x}^{2}}\lesssim&
M^{s_{0}} \n{e^{\lra{v}^{2}}f_{\rm{r}}(t)}_{L_{v}^{2}L_{x}^{2}}\lesssim \exp[-t M^{1-s}],\\
M^{s_{0}-1} \n{e^{\lra{v}^{2}}\nabla_{x}f_{\rm{r}}(t)}_{L_{v}^{2,1}L_{x}^{2}}\lesssim&
M^{s_{0}-1} \n{e^{\lra{v}^{2}}\nabla_{x}f_{\rm{r}}(t)}_{L_{v}^{2}L_{x}^{2}}\lesssim\exp[-t M^{1-s}].
 \end{align*}

\textbf{The $\Vert e^{\lra{v}^{2}} \cdot \Vert _{L_{v}^{1,1}L_{x}^{\infty}}$ and
$M^{-1}\Vert e^{\lra{v}^{2}}\nabla_{x} \cdot \Vert_{L_{v}^{1,1}L_{x}^{\infty}}$ estimates.}

It suffices to consider the more delicate $M^{-1}\Vert e^{\lra{v}^{2}}\nabla_{x}\cdot\Vert_{L_{v}^{1,1}L_{x}^{\infty}}$ norm, as the $\Vert e^{\lra{v}^{2}} \cdot \Vert _{L_{v}^{1,1}L_{x}^{\infty}}$ norm can be dealt with in a similar way.
By the Leibniz rule, we have
\begin{align*}
&\n{e^{\lra{v}^{2}}\nabla_{x}f_{\rm{r}}(t)}_{L_{v}^{1,1}L_{x}^{\infty}}\lesssim \mathrm{I}_{1}+\mathrm{I}_{2},
\end{align*}
where
\begin{align*}
\mathrm{I}_{1}=& M^{1+\frac{3}{2}-s_{0}}N^{\frac{3}{2}}\bbn{\exp\lrc{-\int_{0}^{t}A[f_{\mathrm{b}}](\tau,x,v)d\tau}(\nabla\chi)(100Mx)\chi(Nv)}_{L_{v}^{1,1}L_{x}^{\infty}},\\
\mathrm{I}_{2}=&M^{\frac{3}{2}-s_{0}}N^{\frac{3}{2}}\bbn{\lrc{\nabla_{x}\int_{0}^{t}
A[f_{\mathrm{b}}](\tau)d\tau}\exp\lrc{-\int_{0}^{t}A[f_{\mathrm{b}}](\tau)d\tau}
\chi(100Mx)\chi(Nv)}_{L_{v}^{1,1}L_{x}^{\infty}}.
\end{align*}
Combining the pointwise upper bound \eqref{equ:upper bound,fr,k} and the lower bound \eqref{equ:upper lower bound,fr} for $A[f_{\mathrm{b}}]$ with the elementary estimates
$\n{(\nabla^{k}\chi)(Mx)}_{L_{x}^{\infty}}\lesssim 1$
and $\n{\chi(Nv)}_{L_{v}^{1,1}}\lesssim N^{-3}$, we obtain
\begin{align}\label{equ:fr,Z-norm,L1}
&M^{-1}\n{e^{\lra{v}^{2}}\nabla_{x}f_{\rm{r}}(t)}_{L_{v}^{1,1}L_{x}^{\infty}}\\
\lesssim& M^{-1}(\mathrm{I}_{1}+\mathrm{I}_{2})\notag\\
\lesssim& M^{-1}M^{1+\frac{3}{2}-s_{0}}N^{\frac{3}{2}}\exp[-2tM^{1-s}]\n{(\nabla\chi)(100Mx)}_{L_{x}^{\infty}}\n{\chi(Nv)}_{L_{v}^{1,1}}\notag\\
&+M^{-1}M^{\frac{3}{2}-s_{0}}N^{\frac{3}{2}}\lra{tM^{1+1-s}}\exp[-2tM^{1-s}]\n{\chi(100Mx)}_{L_{x}^{\infty}}\n{\chi(Nv)}_{L_{v}^{1,1}}\notag\\
	\lesssim&  M^{\frac{3}{2}-s_{0}}N^{-\frac{3}{2}}\exp[-2t M^{1-s}]\lra{tM^{1-s}}\lesssim N^{-1},\notag
\end{align}
where in the last inequality we used
the fact that $N=M^{10}$ to absorb the remaining powers of $M$.
Therefore, we have completed the proof of the $Z$-norm estimate \eqref{equ:z-norm bounds for fr} for $f_{\mathrm{r}}$.

Finally, the $Z$-norm estimate for $f_{\mathrm{a}}$ follows directly from the triangle inequality
\begin{align*}
	\n{e^{\lra{v}^{2}}f_{\mathrm{a}}(t)}_{Z}\lesssim \n{e^{\lra{v}^{2}}f_{\mathrm{r}}(t)}_{Z}+\n{e^{\lra{v}^{2}}f_{\mathrm{b}}(t)}_{Z}.
\end{align*}
Then combining the $Z$-norm estimate \eqref{equ:z-norm estimate for fb} for $f_{\mathrm{b}}$
and $Z$-norm estimate \eqref{equ:z-norm bounds for fr} for $f_{\mathrm{r}}$, we complete the proof of the $Z$-norm estimate \eqref{equ:z-norm estimate for fa} for $f_{\mathrm{a}}$.

\end{proof}

\subsection{Analysis of $v\cdot \nabla_{x}f_{\mathrm{r}}$} \label{section:Analysis of term1}
We now address the $Z$-norm estimate for the transport term $v \cdot \nabla_{x}f_{\mathrm{r}}$ in the following lemma.
\begin{lemma}
For $0\leq  t_{0}\leq t \leq T_{*}$, we have
\begin{align*}
\bbn{ e^{\lra{v}^{2}}\int_{t_{0}}^{t}e^{-(t-\tau)v\cdot \nabla_{x}}(v\cdot \nabla_{x}f_{\mathrm{r}})(\tau)\,d\tau} _{Z}\lesssim M^{-\delta}.
\end{align*}
\end{lemma}
\begin{proof}
Under the choice $N=M^{10}$ stated in \eqref{equ:condition,parameters}, it is sufficient to establish a decay estimate that carries a smallness factor of the form $N^{-\delta}$, which immediately implies the desired decay $M^{-\delta}$. We verify this by estimating the components of the $Z$-norm separately.
	
\textbf{The $M^{s_{0}}\Vert e^{\lra{v}^{2}} \cdot \Vert_{L_{v}^{2,1}L_{x}^{2}}$ and $M^{s_{0}-1}\Vert e^{\lra{v}^{2}}\nabla_{x}\cdot\Vert_{L_{v}^{2,1}L_{x}^{2}}$
 estimates.}

It suffices to treat the $M^{s_{0}-1}\Vert e^{\lra{v}^{2}}\nabla_{x}\cdot\Vert_{L_{v}^{2,1}L_{x}^{2}}$ norm, as the estimate for the $M^{s_{0}}\Vert e^{\lra{v}^{2}} \cdot \Vert_{L_{v}^{2,1}L_{x}^{2}}$ norm follows analogously.
Since $f_{\mathrm{r}}$ is supported where $\lr{|v|\lesssim N^{-1}}$, we have
\begin{align}\label{equ:v term,H1}
\n{e^{\lra{v}^{2}}\nabla_{x} (v\cdot \nabla_{x} f_{\mathrm{r}})}_{L_{v}^{2,1}L_{x}^{2}}\lesssim&
N^{-1}\n{\Delta_{x}f_{\mathrm{r}}}_{L_{v}^{2}L_{x}^{2}}
\lesssim N^{-1}M^{2-s_{0}},
\end{align}
where the last inequality utilizes the bounds from the proof of \eqref{equ:upper bound,fr,sobolev} with one $x$-derivative producing a factor of $M$.
We then insert in $T_{*}=M^{s-1}(\ln\ln M)$ to get
\begin{align*}
	&\bbn{e^{\lra{v}^{2}}\nabla_{x} \int_{t_{0}}^{t}e^{-(t-\tau)v\cdot \nabla _{x}}(v\cdot \nabla_{x}f_{\mathrm{r}})(\tau)\,d\tau}_{L_{v}^{2,1}L_{x}^{2}}\\
	\lesssim& T_{*} \sup_{t\in[0,T_{*}]}\n{e^{\lra{v}^{2}}\nabla_{x}(v\cdot \nabla_{x} f_{\mathrm{r}})}_{L_{v}^{2,1}L_{x}^{2}}\\
	\lesssim&M^{s-1}(\ln\ln M) N^{-1}M^{2-s_{0}}\\
\lesssim&  N^{-\frac{1}{2}},
\end{align*}
which suffices for our goal.

\textbf{The $\Vert e^{\lra{v}^{2}} \cdot \Vert _{L_{v}^{1,1}L_{x}^{\infty}}$ and $M^{-1}\Vert e^{\lra{v}^{2}}\nabla_{x}\cdot\Vert _{L_{v}^{1,1}L_{x}^{\infty}}$ estimates.}

Again, it suffices to treat the $M^{-1}\Vert e^{\lra{v}^{2}}\nabla_{x}\cdot\Vert _{L_{v}^{1,1}L_{x}^{\infty}}$ norm, as the $\Vert e^{\lra{v}^{2}} \cdot \Vert _{L_{v}^{1,1}L_{x}^{\infty}}$ norm can be dealt with in a similar way.
We recall that
\begin{equation*}
f_{\mathrm{r}}(t,x,v)=M^{\frac{3}{2}-s_{0}}N^{\frac{3}{2}}\exp \left[
-\int_{0}^{t}A[f_{\mathrm{b}}](\tau,x,v)d\tau\right] \chi (100Mx)\chi (Nv).
\end{equation*}
Using the support condition $|v|\lesssim N^{-1}$, and Leibniz rule, we have
\begin{align*}
&\n{e^{\lra{v}^{2}}\nabla_{x}(v\cdot \nabla_{x}f_{\mathrm{r}})}_{L_{v}^{1,1}L_{x}^{\infty}}
\lesssim N^{-1}M^{\frac{3}{2}-s_{0}}N^{\frac{3}{2}}\lrs{\mathrm{I}_{1}+\mathrm{I}_{2}+\mathrm{I}_{3}+\mathrm{I}_{4}},
\end{align*}
where
\begin{align*}
\mathrm{I}_{1}=&M^{2}\bbn{\exp\lrc{-\int_{0}^{t}A[f_{\mathrm{b}}](\tau)d\tau}(\nabla^{2} \chi)(100Mx)\chi(Nv)}_{L_{v}^{1,1}L_{x}^{\infty}},\\ \mathrm{I}_{2}=&M\bbn{\lrc{\nabla_{x}\int_{0}^{t}A[f_{\mathrm{b}}](\tau)d\tau}
\exp\lrc{-\int_{0}^{t}A[f_{\mathrm{b}}](\tau)d\tau}(\nabla \chi)(100Mx)\chi(Nv)}_{L_{v}^{1,1}L_{x}^{\infty}},\\
\mathrm{I}_{3}=&\bbn{
\lrc{\nabla_{x}^{2}\int_{0}^{t}A[f_{\mathrm{b}}](\tau)d\tau}
\exp\lrc{-\int_{0}^{t}A[f_{\mathrm{b}}](\tau)d\tau} \chi(100Mx)\chi(Nv)}_{L_{v}^{1,1}L_{x}^{\infty}},\\
\mathrm{I}_{4}=&\bbn{\lrc{\nabla_{x}\int_{0}^{t}A[f_{\mathrm{b}}](\tau)d\tau}^{2}
\exp\lrc{-\int_{0}^{t}A[f_{\mathrm{b}}](\tau)d\tau} \chi (100Mx)\chi(Nv)}_{L_{v}^{1,1}L_{x}^{\infty}}.
\end{align*}
We note that the pointwise upper bound \eqref{equ:upper bound,fr,k}
\begin{align}\label{equ:upper bound,fr,k,used}
\big|\chi(Nv) \nabla_{x}^{k}A[f_{\mathrm{b}}](t,x,v)\big |\lesssim M^{k+1-s}
\end{align}
and the pointwise lower bound \eqref{equ:upper lower bound,fr}
\begin{equation}\label{equ:upper lower bound,fr,used}
\chi(Nv)A[f_{\mathrm{b}}] \gtrsim  M^{1-s} \chi(Nv),\quad \text{for $|x|\leq \frac{1}{100M}$}.
\end{equation}
Together with the basic estimates
$\n{(\nabla^{k}\chi)(Mx)}_{L_{x}^{\infty}}\lesssim 1$
and $\n{\chi(Nv)}_{L_{v}^{1,1}}\lesssim N^{-3}$, we arrive at
\begin{align*}
&N^{-1}M^{\frac{3}{2}-s_{0}}N^{\frac{3}{2}}\lrs{\mathrm{I}_{1}+\mathrm{I}_{2}+\mathrm{I}_{3}
+\mathrm{I}_{4}}\\
\lesssim& N^{-1}M^{\frac{3}{2}-s_{0}}N^{\frac{3}{2}}M^{2}N^{-3}\exp[-2tM^{1-s}]\lra{tM^{1-s}}^{2}\\
\lesssim& N^{-\frac{5}{2}}M^{\frac{7}{2}-s_{0}}.
\end{align*}
Multiplying by $T_{*}=M^{s-1}(\ln\ln M)$ and using $N=M^{10}$ gives
\begin{align}\label{equ:vfr,L1}
&M^{-1} \bbn{e^{\lra{v}^{2}}\nabla_{x} \int_{t_{0}}^{t}e^{-(t-\tau)v\cdot \nabla _{x}}(v\cdot \nabla_{x}f_{\mathrm{r}})(\tau)\,d\tau}_{L_{v}^{1,1}L_{x}^{\infty}}\\
\lesssim& T_{*}N^{-\frac{5}{2}}M^{\frac{7}{2}-s_{0}}
\lesssim N^{-1}, \notag
\end{align}
which is sufficient for our goal.

\end{proof}

\subsection{Analysis of $Q^{+}(f_{\mathrm{r}},f_{\mathrm{b}})$, $Q^{\pm}(f_{\mathrm{b}},f_{\mathrm{r}})$, and
$Q^{\pm}(f_{\mathrm{r}},f_{\mathrm{r}})$}\label{section:Analysis of term2}
We begin by recalling the key bounds for $f_{\mathrm{b}}$ and $f_{\mathrm{r}}$ established in Lemma \ref{lemma:z-norm bounds on,fb,fr,fa}. For the time interval $[0, T_{*}]$, we have
\begin{align}
\n{e^{\lra{v}^{2}}f_{\mathrm{b}}}_{L_{t}^{\infty}(0,T_{*};Z)}\lesssim& M^{1-s}, \label{equ:z-norm,fb,used}\\
\n{e^{\lra{v}^{2}}f_{\mathrm{r}}}_{L_{t}^{\infty}(0,T_{*};Z)}\lesssim& 1,\label{equ:z-norm,fr,used}\\
\n{e^{\lra{v}^{2}}f_{\mathrm{r}}}_{L_{t}^{\infty}(0,T_{*};L_{v}^{1,1}L_{x}^{\infty})}\lesssim & N^{-1},\label{equ:z-norm,lv1,used}\\
M^{-1}\n{e^{\lra{v}^{2}}\nabla_{x}f_{\mathrm{r}}}_{L_{t}^{\infty}(0,T_{*};L_{v}^{1,1}L_{x}^{\infty})}
\lesssim& N^{-1},\label{equ:z-norm,lv1,H1,used}
\end{align}
where the last two decay estimates \eqref{equ:z-norm,lv1,used}--\eqref{equ:z-norm,lv1,H1,used} follow from estimate \eqref{equ:fr,Z-norm,L1}.

We now proceed to the analysis of the bilinear collision terms. Applying the time-independent, weighted $Z$-norm collision estimates \eqref{equ:Z,norm,bilinear,estimate,time-independent}--\eqref{equ:Z,norm,bilinear,estimate,time-independent,gZ} from Proposition \ref{lemma:Z,norm,bilinear,estimate} yields
\begin{align*}
\n{e^{\lra{v}^{2}}Q^{-}(f_{\mathrm{b}},f_{\mathrm{r}})}_{Z}\lesssim& \n{\lra{v}e^{\lra{v}^{2}}f_{\mathrm{b}}}_{Z}\lrs{\n{\lra{v}e^{\lra{v}^{2}}f_{\mathrm{r}}}_{L_{v}^{1,1}L_{x}^{\infty}}
+M^{-1}\n{\lra{v}e^{\lra{v}^{2}}\nabla_{x}f_{\mathrm{r}}}_{L_{v}^{1,1}L_{x}^{\infty}}},\\
\n{e^{\lra{v}^{2}}Q^{+}(f_{\mathrm{b}},f_{\mathrm{r}})}_{Z}\lesssim & \n{\lra{v}e^{\lra{v}^{2}}f_{\mathrm{b}}}_{Z}\lrs{\n{\lra{v}e^{\lra{v}^{2}}f_{\mathrm{r}}}_{L_{v}^{1,1}L_{x}^{\infty}}
+M^{-1}\n{\lra{v}e^{\lra{v}^{2}}\nabla_{x}f_{\mathrm{r}}}_{L_{v}^{1,1}L_{x}^{\infty}}},\\
\n{e^{\lra{v}^{2}}Q^{+}(f_{\mathrm{r}},f_{\mathrm{b}})}_{Z}\lesssim & \n{\lra{v}e^{\lra{v}^{2}}f_{\mathrm{b}}}_{Z}\lrs{\n{\lra{v}e^{\lra{v}^{2}}f_{\mathrm{r}}}_{L_{v}^{1,1}L_{x}^{\infty}}
+M^{-1}\n{\lra{v}e^{\lra{v}^{2}}\nabla_{x}f_{\mathrm{r}}}_{L_{v}^{1,1}L_{x}^{\infty}}},\\
\n{e^{\lra{v}^{2}}Q^{\pm}(f_{\mathrm{r}},f_{\mathrm{r}})}_{Z}\lesssim & \n{\lra{v}e^{\lra{v}^{2}}f_{\mathrm{r}}}_{Z}\lrs{\n{\lra{v}e^{\lra{v}^{2}}f_{\mathrm{r}}}_{L_{v}^{1,1}L_{x}^{\infty}}
+M^{-1}\n{\lra{v}e^{\lra{v}^{2}}\nabla_{x}f_{\mathrm{r}}}_{L_{v}^{1,1}L_{x}^{\infty}}}.
\end{align*}
Although these estimates exhibit a loss of velocity weight $\lra{v}$, this extra factor can be removed due to the uniform compact support of the approximate solutions $f_{\mathrm{r}}$ and $f_{\mathrm{b}}$.

\begin{lemma}
For $0\leq  t_{0}\leq t \leq T_{*}$, we have
\begin{align*}
\bbn{e^{\lra{v}^{2}} \int_{t_{0}}^{t}e^{-(t-\tau)v\cdot \nabla_{x}}\lrs{Q^{+}(f_{\mathrm{r}},f_{\mathrm{b}})\pm Q^{\pm}(f_{\mathrm{b}},f_{\mathrm{r}})\pm Q^{\pm}(f_{\mathrm{r}},f_{
\mathrm{r}})}(\tau)d\tau}_{Z}\lesssim M^{-\delta}.
\end{align*}
\end{lemma}
\begin{proof}
For any triple $(\mathrm{sgn}, i, j)$ satisfying
$$(\mathrm{sgn}, i, j)\in\lr{(+,r,b),(\pm,b,r),(\pm,r,r)},$$
using the bounds \eqref{equ:z-norm,fb,used}--\eqref{equ:z-norm,lv1,H1,used}, we obtain
\begin{align*}
&\bbn{e^{\lra{v}^{2}} \int_{t_{0}}^{t}e^{-(t-\tau)v\cdot \nabla_{x}}Q^{sgn}(f_{i},f_{j})(\tau)\,d\tau}_{Z}\\
\lesssim& T_{*}\n{e^{\lra{v}^{2}}Q^{sgn}(f_{i},f_{j})}_{L_{t}^{\infty}(0,T_{*};Z)}\\
\lesssim& T_{*}\lrs{\n{\lra{v}e^{\lra{v}^{2}}f_{\mathrm{r}}}_{Z}+\n{\lra{v}e^{\lra{v}^{2}}f_{\mathrm{b}}}_{Z}}
\lrs{\n{\lra{v}e^{\lra{v}^{2}}
f_{\mathrm{r}}}_{L_{v}^{1,1}L_{x}^{\infty}}
+M^{-1}\n{\lra{v}e^{\lra{v}^{2}}\nabla_{x}f_{\mathrm{r}}}_{L_{v}^{1,1}L_{x}^{\infty}}}\\
\lesssim& T_{*}M^{1-s} N^{-1}\lesssim (\ln\ln M)N^{-1},
\end{align*}
where in the last inequality we have inserted in $T_{*}=M^{s-1}(\ln\ln M)$. Given that $N=M^{10}$, this estimate is sufficient for our purpose.

\end{proof}

\subsection{Analysis of $Q^{-}(f_{\mathrm{b}},f_{\mathrm{b}})$} \label{section:Analysis of Q-}
We get into the analysis of the $Z$-norm estimate for the loss term $Q^{-}(f_{\mathrm{b}},f_{\mathrm{b}})$. It suffices to consider the
$M^{s_{0}}\Vert e^{\lra{v}^{2}} \cdot \Vert_{L_{v}^{2,1}L_{x}^{2}}$ and $\Vert e^{\lra{v}^{2}} \cdot \Vert _{L_{v}^{1,1}L_{x}^{\infty}}$ estimates.
The estimates involving the $x$-derivative terms, i.e., $M^{s_{0}-1}\n{\nabla_{x}f}_{L_{v}^{2,1}L_{x}^{2}}$ and $M^{-1}\n{\nabla_{x}f}_{L_{v}^{1,1}L_{x}^{\infty}}$, follow from an analogous argument.
 This is due to the fact that each spatial derivative $\nabla_{x}$ contributes an additional factor of $M$, which is already compensated by the factor
$M^{-1}$ present in the corresponding norms.

\begin{lemma}
For $0\leq  t_{0}\leq t \leq T_{*}$, we have
	\begin{align}\label{equ:Z norm,Q-,fb,fb}
		\bbn{ e^{\lra{v}^{2}}\int_{t_{0}}^{t}e^{-(t-\tau)v\cdot \nabla _{x}}Q^{-}(f_{\mathrm{b}},f_{\mathrm{b}})(\tau)\,d\tau} _{Z}\lesssim M^{-\delta}.
	\end{align}
\end{lemma}
\begin{proof}

\textbf{The $M^{s_{0}}\Vert e^{\lra{v}^{2}} \cdot \Vert_{L_{v}^{2,1}L_{x}^{2}}$ estimate.}

Due to the uniform compact support of $f_{\mathrm{b}}$, the Gaussian weight $e^{\lra{v}^{2}}$ and velocity weight $\lra{v}$ are uniformly bounded and can be discarded. Therefore, we have
\begin{equation}\label{equ:equ:Q-,fb,fb,pointwise,first}
	M^{s_{0}}\bbn{ e^{\lra{v}^{2}}\int_{t_{0}}^{t}e^{-(t-\tau)v\cdot \nabla _{x}}Q^{-}(f_{\mathrm{b}},f_{\mathrm{b}})(\tau)\,d\tau}_{L_{v}^{2,1}L_{x}^{2}}\lesssim T_{*}M^{s_{0}}\n{Q^{-}(f_{\mathrm{b}},f_{\mathrm{b}})}_{L_{t}^{\infty }(0,T_{*};L_{v}^{2}L_{x}^{2})}.
\end{equation}
Applying the collision estimate \eqref{equ:collision,estimate,Q-,Lr} and H\"{o}lder inequality, we obtain
\begin{align*}
M^{s_{0}}\n{Q^{-}(f_{\mathrm{b}},f_{\mathrm{b}})}_{L_{t}^{\infty }(0,T_{*};L_{x}^{2}L_{v}^{2})}
\lesssim& M^{s_{0}}\n{f_{\mathrm{b}}}_{L_{t}^{\infty }(0,T_{*};L_{x}^{\infty}L_{v}^{2,1})}
\n{f_{\mathrm{b}}}_{L_{t}^{\infty }(0,T_{*};L_{x}^{2}L_{v}^{1,1})}.
\end{align*}
Invoking Sobolev norm bounds \eqref{equ:sobolev norm,fb,LvpLxinf} and \eqref{equ:sobolev norm,fb,L2L1} for $f_{\mathrm{b}}$ in Lemma \ref{lemma:Hs,bounds on fb}, we arrive at
\begin{align*}
M^{s_{0}}\bbn{ e^{\lra{v}^{2}}\int_{t_{0}}^{t}e^{-(t-\tau)v\cdot \nabla _{x}}Q^{-}(f_{\mathrm{b}},f_{\mathrm{b}})(\tau)\,d\tau}_{L_{v}^{2}L_{x}^{2}}\lesssim& T_{*}M^{s_{0}}M^{1-s}M^{-\frac{1}{2}-s}\\
\lesssim& M^{-\frac{1}{2}+s_{0}-s}(\ln\ln M),
\end{align*}
which suffices for our goal with $s>s_{0}$.

\textbf{The $\Vert e^{\lra{v}^{2}} \cdot \Vert _{L_{v}^{1,1}L_{x}^{\infty}}$ estimate.}

For notational convenience, we define
\begin{equation*}
	D^{-}=\int_{t_{0}}^{t}e^{-(t-\tau)v\cdot \nabla
		_{x}}Q^{-}(f_{\mathrm{b}},f_{\mathrm{b}})(\tau)d\tau.
\end{equation*}
Using the pointwise estimates \eqref{equ:fb,pointwise,estimate,upper} and \eqref{equ:upper bound,Kj,sum} for $f_{\mathrm{b}}$, we get
\begin{align*}
f_{\mathrm{b}}(t)\lesssim& M^{1-s}\sum_{i,j}\wt{K}_{i,j}(x)I_{i,j}(v),\\
A[f_{\mathrm{b}}](t)\lesssim& M^{1-s}M^{-2}\sum_{i,j}\wt{K}_{i,j}(x)\lesssim M^{1-s}\lrs{\frac{1}{M|x|+1}}^{2},
\end{align*}
where $\wt{K}_{ij}(x)=\chi(\frac{MP_{e_{i,j}}^{\perp }x}{10})\chi (\frac{P_{e_{i,j}}x}{10})$.
Expanding $D^{-}$ gives that
\begin{align}\label{equ:D-,first,expansion}
D^{-}=&\int_{t_{0}}^{t}Q^{-}(f_{\mathrm{b}},f_{\mathrm{b}})(\tau,x-v(t-\tau),v)d\tau\\
\lesssim &\int_{t_{0}}^{t} \lrs{M^{1-s}}^{2} \sum_{i,j}\wt{K}_{i,j}(x-v(t-\tau))I_{i,j}(v)  \lrs{\frac{1}{
M|x-v(t-\tau)|+1}}^{2}d\tau.\notag
\end{align}
Using the summation estimate \eqref{equ:Iij,sum} for $I_{i,j}$, we have
\begin{align}\label{equ:D-,Kij,Iij}
\sum_{i,j}\wt{K}_{i,j}(x-v(t-\tau))I_{i,j}(v)\lesssim \sum_{i,j}I_{i,j}(v)\lesssim 1_{\lr{\frac{3}{4}\leq|v|\leq \frac{5}{4}}}(v):=I(v).
\end{align}
Putting \eqref{equ:D-,Kij,Iij} into \eqref{equ:D-,first,expansion} yields
\begin{align}\label{equ:D-,Iv,estimate}
D^{-}\lesssim M^{2-2s}I(v)
\int_{0}^{t}  \lrs{\frac{1}{
M|x-v(t-\tau)|+1}}^{2}d\tau.
\end{align}

It remains to estimate the time integral. Changing variables $\sigma = (t-\tau)$ and noting that $|v|\sim 1$, we obtain
\begin{align}\label{equ:time integral,D-}
	& I(v)\int_{0}^{t}  \lrs{\frac{1}{
		M|x-v(t-\tau)|+1}}^{2}d\tau\\
	\lesssim&I(v)\int_{0}^{T_{*}} \lrs{\frac{1}{|Mx-Mv\sigma|+1}}^{2}
d\sigma\notag\\
	\leq& \frac{I(v)}{M|v|}\int_{0}^{MT_{*}|v|}\lrs{\frac{1}{\babs{Mx-\sigma \frac{v}{|v|}}+1}}^{2}d\sigma \notag\\
	\lesssim& \frac{I(v)}{M},\notag
\end{align}
where in the last step we have used the inequalities
\begin{align*}
\bbabs{Mx-\sigma \frac{v}{|v|}}\geq |M|x|-\sigma|,\quad
\int_{\R} \frac{1}{\lra{M|x|-\sigma}^{2}}d\sigma\lesssim 1.
\end{align*}
Putting \eqref{equ:time integral,D-} into \eqref{equ:D-,Iv,estimate}, we arrive at
\begin{align}\label{equ:Q-,fb,fb,L1,final}
	\n{D^{-}}_{L_{v}^{1,1}L_{x}^{\infty}}\lesssim& M^{2-2s} \frac{\n{I(v)}_{L_{v}^{1,1}}}{M}
	\lesssim M^{1-2s} .
\end{align}
Since $s>\frac{1}{2}$, this term is bounded by a negative power of $M$.
\end{proof}

\subsection{Analysis of $Q^{+}(f_{\mathrm{b}},f_{\mathrm{b}})$} \label{section:Analysis of Q+}
We proceed to analyze the $Z$-norm estimate for the gain term $Q^{+}(f_{\mathrm{b}},f_{\mathrm{b}})$. As in the case of $Q^{-}(f_{\mathrm{b}},f_{\mathrm{b}})$,
our attention can be restricted to the estimates for
$M^{s_{0}}\Vert e^{\lra{v}^{2}} \cdot \Vert_{L_{v}^{2,1}L_{x}^{2}}$ and $\Vert e^{\lra{v}^{2}} \cdot \Vert _{L_{v}^{1,1}L_{x}^{\infty}}$.
The estimates involving spatial derivatives, specifically $M^{s_{0}-1}\n{e^{\lra{v}^{2}}\nabla_{x}\cdot}_{L_{v}^{2,1}L_{x}^{2}}$ and $M^{-1}\n{e^{\lra{v}^{2}}\nabla_{x}\cdot}_{L_{v}^{1,1}L_{x}^{\infty}}$, follow from a similar way, since each derivative $\nabla_{x}$ introduces an additional factor of
$M$, which is precisely compensated by the factor
$M^{-1}$ in the corresponding norms.

\begin{lemma}
For $0\leq  t_{0}\leq t \leq T_{*}$, we have
	\begin{align*}
		\bbn{e^{\lra{v}^{2}}\int_{t_{0}}^{t}e^{-(t-\tau)v\cdot \nabla _{x}}Q^{+}(f_{\mathrm{b}},f_{\mathrm{b}})(\tau)\,d\tau}_{Z}\lesssim M^{-\delta}.
	\end{align*}
\end{lemma}
\begin{proof}

\textbf{The $M^{s_{0}}\Vert e^{\lra{v}^{2}} \cdot \Vert_{L_{v}^{2,1}L_{x}^{2}}$ estimate.}

In view of the energy conservation law, which implies
\begin{align}\label{equ:energy conservation,inequality}
\lra{v}\lesssim \lra{v^{*}}+\lra{u^{*}},\quad e^{\lra{v}^{2}}\lesssim e^{\lra{v^{*}}^{2}+\lra{u^{*}}^{2}},
\end{align}
we deduce
\begin{align}\label{equ:equ:Q+,fb,fb,pointwise,first}
	&M^{s_{0}}\bbn{ e^{\lra{v}^{2}}\int_{t_{0}}^{t}e^{-(t-\tau)v\cdot \nabla _{x}}Q^{+}(f_{\mathrm{b}},f_{\mathrm{b}})(\tau)\,d\tau}_{L_{v}^{2,1}L_{x}^{2}}\\
\lesssim& T_{*}M^{s_{0}}\n{Q^{+}(\lra{v}e^{\lra{v}^{2}}f_{\mathrm{b}},\lra{v}e^{\lra{v}^{2}}f_{\mathrm{b}})}_{L_{t}^{\infty }(0,T_{*};L_{v}^{2}L_{x}^{2})}\notag
\end{align}
Upon applying the collision estimate \eqref{equ:collision,estimate,Q+,Lr} and H\"{o}lder inequality, we obtain
\begin{align*}
&M^{s_{0}}\n{Q^{+}(\lra{v}e^{\lra{v}^{2}}f_{\mathrm{b}},\lra{v}e^{\lra{v}^{2}}f_{\mathrm{b}})}_{L_{t}^{\infty }(0,T_{*};L_{x}^{2}L_{v}^{2})}\\
\lesssim& M^{s_{0}}\n{\lra{v}e^{\lra{v}^{2}}f_{\mathrm{b}}}_{L_{t}^{\infty }(0,T_{*};L_{x}^{\infty}L_{v}^{2,1})}
\n{\lra{v}e^{\lra{v}^{2}}f_{\mathrm{b}}}_{L_{t}^{\infty }(0,T_{*};L_{x}^{2}L_{v}^{1,1})}\\
\lesssim& M^{s_{0}}\n{f_{\mathrm{b}}}_{L_{t}^{\infty }(0,T_{*};L_{x}^{\infty}L_{v}^{2})}
\n{f_{\mathrm{b}}}_{L_{t}^{\infty }(0,T_{*};L_{x}^{2}L_{v}^{1})},
\end{align*}
where in the last inequality we have utilized the uniform compact support of $f_{\mathrm{b}}$ to absorb the weight function $\lra{v}e^{\lra{v}^{2}}$.
Then, applying the Sobolev norm estimates \eqref{equ:sobolev norm,fb,LvpLxinf} and \eqref{equ:sobolev norm,fb,L2L1} for $f_{\mathrm{b}}$ from Lemma \ref{lemma:Hs,bounds on fb}, and substituting $T_{*}=M^{s-1}(\ln\ln M)$, we arrive at
\begin{align*}
M^{s_{0}}\bbn{ e^{\lra{v}^{2}}\int_{t_{0}}^{t}e^{-(t-\tau)v\cdot \nabla _{x}}Q^{+}(f_{\mathrm{b}},f_{\mathrm{b}})(\tau)\,d\tau}_{L_{v}^{2}L_{x}^{2}}\lesssim& T_{*}M^{s_{0}}M^{1-s}M^{-\frac{1}{2}-s}\\
\lesssim& M^{-\frac{1}{2}+s_{0}-s}(\ln\ln M),
\end{align*}
which is sufficient for our purposes given that $s>s_{0}$.

\textbf{The $\Vert e^{\lra{v}^{2}} \cdot \Vert _{L_{v}^{1,1}L_{x}^{\infty}}$ estimate.}

For brevity, we introduce the notation
\begin{equation*}
	D^{+}=\int_{t_{0}}^{t}e^{-(t-\tau)v\cdot \nabla
		_{x}}Q^{+}(f_{\mathrm{b}},f_{\mathrm{b}})(\tau)d\tau.
\end{equation*}
We recall that $f_{\mathrm{b}}$ takes the form
\begin{align}\label{equ:fb,D+}
f_{\mathrm{b}}(t,x,v)=&M^{1-s}\sum_{i,j}K_{i,j}(x-vt)I_{i,j}(v),
\end{align}
with
\begin{align*}
K_{i,j}(x)=\chi (MP_{e_{i,j}}^{\perp }x)\chi (P_{e_{i,j}}x),\quad I_{i,j}(v)=\chi(MP_{e_{i,j}}^{\perp }v)\chi (10P_{e_{i,j}}(v-e_{i,j})).
\end{align*}
The gain term is defined by
\begin{equation}\label{equ:Q+fg,D+}
Q^{+}(f,g)=\int_{\mathbb{S}^{2}}\int_{\mathbb{R}^{3}}B(u-v,\om)f(v^{*})g(u^{*})\,du\,d\omega,
\end{equation}
where $B(u-v,\omega)=|(u-v)\cdot \omega|$. The pre-collisional and post-collisional velocities are related by
\begin{equation*}
\left\{
\begin{aligned}
v^{*} =&P_{\omega}u+P_{\omega }^{\bot }v,\quad u^{*}=P_{\omega }v+P_{\omega }^{\bot }u, \\
v =&P_{\omega }^{\bot }v^{*}+P_{\omega }u^{*},\quad%
u=P_{\omega }v^{*}+P_{\omega }^{\bot }u^{*}.
\end{aligned}
\right.
\end{equation*}

Substituting \eqref{equ:fb,D+} and \eqref{equ:Q+fg,D+} into $D^{+}$, we get
\begin{align*}
D^{+} =&M^{2-2s}\sum_{i_{1},j_{1}}\sum_{i_{2},j_{2}}\int_{t_{0}}^{t}e^{-(t-\tau)v\cdot \nabla_{x}}Q^{+}(K_{i_{1},j_{1}}(x- v\tau )I_{i_{1},j_{1}}(v),K_{i_{2},j_{2}}(x-v\tau)I_{i_{2},j_{2}}(v))d\tau\\
=&M^{2-2s}\sum_{i_{1},j_{1}}\sum_{i_{2},j_{2}}\int_{
\mathbb{S}^{2}} \int_{\R^{3}} B(u-v,\omega) S_{i_{1},j_{1},i_{2},j_{2}}(t_{0},t,x,\omega,u^{*},v^{*})dud\omega,\notag
\end{align*}
where
\begin{align}\label{equ:Sij,t,x}
&S_{i_{1},j_{1},i_{2},j_{2}}(t_{0},t,x,\omega,u^{*},v^{*})\\
=&\int_{t_{0}}^{t}  K_{i_{1},j_{1}}(x-v(t-\tau)-v^{*}\tau)I_{i_{1},j_{1}}(v^{*}) K_{i_{2},j_{2}}(x-v(t-\tau)-u^{*}\tau)I_{i_{2},j_{2}}(u^{*})d\tau.\notag
\end{align}

We now proceed to establish the estimate for $D^{+}$.
Using the energy inequality \eqref{equ:energy conservation,inequality} and the uniform compact support in the $v^{*}$ and $u^{*}$ variable, we drop the weight function $\lra{v}e^{\lra{v}^{2}}$ to obtain
\begin{align*}
&\n{e^{\lra{v}^{2}}D^{+}}_{L_{v}^{1,1}L_{x}^{\infty }} \\
\lesssim & M^{2-2s}\bbn{\sum_{i_{1},j_{1}}\sum_{i_{2},j_{2}}\int_{
\mathbb{S}^{2}} \int_{\R^{3}} e^{\lra{v}^{2}}B(u-v,\omega)  S_{i_{1},j_{1},i_{2},j_{2}}(t_{0},t,x,\omega,u^{*},v^{*})dud\omega}_{L_{v}^{1,1}L_{x}^{\infty }}\notag \\
\lesssim &M^{2-2s}\sum_{i_{1},j_{1}}\sum_{i_{2},j_{2}}\int_{
\mathbb{S}^{2}} \int_{\R^{3}\times \R^{3}} B(u-v,\omega)\left\Vert S_{i_{1},j_{1},i_{2},j_{2}}(t_{0},t,x,\omega,u^{*},v^{*})\right\Vert_{L_{x}^{\infty }} du dv d\omega\notag
\end{align*}
Subsequently, applying the collision transform $(u^{*},v^{*})\mapsto (u,v)$ and utilizing the collision inequality
$$B(u^{*}-v^{*},\omega)=|(u^{*}-v^{*})\cdot \omega|\leq|u^{*}-v^{*}|=|u-v|,$$
we derive
\begin{align}\label{equ:estimate,D+}
&\n{e^{\lra{v}^{2}}D^{+}}_{L_{v}^{1,1}L_{x}^{\infty }}\\
\lesssim&M^{2-2s}\sum_{i_{1},j_{1}}\sum_{i_{2},j_{2}}\int_{
\mathbb{S}^{2}} \int_{\R^{3}\times \R^{3}} B(u^{*}-v^{*},\omega) \left\Vert S_{i_{1},j_{1},i_{2},j_{2}}(t_{0},t,x,\omega,u,v)\right\Vert_{L_{x}^{\infty }}du dv  d\omega,\notag\\
\lesssim &M^{2-2s}\sum_{i_{1},j_{1}}\sum_{i_{2},j_{2}}\int_{\R^{3}\times \R^{3}}
|u-v|\lrc{\int_{\mathbb{S}^{2}} \n{S_{i_{1},j_{1},i_{2},j_{2}}
(t_{0},t,x,\omega,u,v)}_{L_{x}^{\infty}}d\omega}dudv.\notag
\end{align}

We turn to the analysis of the interaction term $S_{i_{1},j_{1},i_{2},j_{2}}(t_{0},t,x,\omega,u,v)$.
Observing the geometric relations
\begin{equation*}
v^{*}-v=-P_{\omega}(v-u),\quad v^{*}-u=P_{\omega }^{\bot}(v-u),
\end{equation*}
we deduce
\begin{equation}\label{equ:Sij,x,v}
\left\{
\begin{aligned}
&x-v^{*}(t-\tau)-v\tau =x-v^{*}t-P_{\omega}(v-u)\tau,\\
&x-v^{*}(t-\tau)-u\tau =x-v^{*}t+P_{\omega}^{\bot}(v-u)\tau.
\end{aligned}
\right.
\end{equation}
Substituting \eqref{equ:Sij,x,v} into the definition \eqref{equ:Sij,t,x} of $S_{i_{1},j_{1},i_{2},j_{2}}(t_{0},t,x,\omega,u,v)$, we have
\begin{align}\label{equ:Sjk,estimate}
&S_{i_{1},j_{1},i_{2},j_{2}}(t_{0},t,x,\omega,u,v) \\
\lesssim &\int_{0}^{T_{*}}\chi \left( MP_{e_{i_{1},j_{1}}}^{\perp }(x-v^{*}t-P_{\omega}(v-u)\tau)\right)
\chi(P_{e_{i_{1},j_{1}}}(x-v^{*}t-P_{\omega }(v-u)\tau))
I_{i_{1},j_{1}}(v)\notag\\
&\chi\left( MP_{e_{i_{2},j_{2}}}^{\perp}(x-v^{*}t+P_{\omega}^{\bot}(v-u)\tau)\right) \chi (P_{e_{i_{2},j_{2}}}(x-v^{*}t+P_{\omega}^{\bot}(v-u)\tau)) I_{i_{2},j_{2}}(u)d\tau\notag\\
\leq &I_{i_{1},j_{1}}(v) I_{i_{2},j_{2}}(u)E_{i_{1},j_{1}}(t,x,\omega,u,v),\notag
\end{align}
where
\begin{align*}
E_{i_{1},j_{1}}(t,x,\omega,u,v)
=:&\int_{0}^{T_{*}}\chi \left( MP_{e_{i_{1},j_{1}}}^{\perp }(x-v^{*}t-P_{\omega}(v-u)\tau)\right) d\tau.
\end{align*}

We now distinguish two cases.

\textbf{Case $1$: $(i_{1},j_{1})\neq (i_{2},j_{2})$.}

We first analyse $E_{i_{1},j_{1}}(t,x,\omega,u,v)$.
First, it admits the trivial upper bound
\begin{align}\label{equ:E,trivial bound}
E_{i_{1},j_{1}}(t,x,\omega,u,v)\leq T_{*}\leq  1.
\end{align}
However, this trivial bound is insufficient to yield the desired decay estimate. Therefore, we seek a more refined estimate.

To proceed, we note the following time-integration estimate
\begin{align}\label{equ:chi,inequality}
\int_{\R}\chi(\overrightarrow{n}\tau+\overrightarrow{m})d\tau=
\frac{1}{|\overrightarrow{n}|}
\int_{\R}\chi\lrs{\frac{\overrightarrow{n}}{|\overrightarrow{n}|}\tau+\overrightarrow{m}}d\tau
\lesssim \frac{1}{|\overrightarrow{n}|},
\end{align}
This time-integration estimate effectively reduces the bound to a geometric constraint involving the projection directions.
By \eqref{equ:chi,inequality}, we arrive at
\begin{align}\label{equ:E,nontrivial bound}
E_{i_{1},j_{1}}(t,x,\omega,u,v)
=&\int_{0}^{T_{*}}\chi \left( MP_{e_{i_{1},j_{1}}}^{\perp }(x-v^{*}t-P_{\omega}(v-u)\tau)\right)d\tau\\
\lesssim& \frac{1}{M|P_{e_{i_{1},j_{1}}}^{\perp}P_{\omega
}(v-u)|}\notag\\
\lesssim&  \frac{1}{M\sin \phi_{i_{1},j_{1}}(\omega)}
\frac{1}{|P_{\omega
}(v-u)|}\notag\\
\lesssim & \frac{1}{M\sin \phi_{i_{1},j_{1}}(\omega)} \frac{1}{\cos \theta(v,u,\omega)|v-u|},\notag
\end{align}
where $\phi_{i_{1},j_{1}}(\omega)\in [0,\pi)$ is the angle between $\omega$ and $e_{i_{1},j_{1}}$, and $\theta(v,u,\omega)\in [0,\pi)$ be the angle between $\omega$ and $v-u$.

Next, we deal with the relative velocity term $|v-u|$. Indeed,
based on the localized properties induced by the support of $I_{i_{1},j_{1}}(v)I_{i_{2},j_{2}}(u)$, we have already established the lower bound in \eqref{equ:lower bound,v,u,eij} for the relative velocity
\begin{align}\label{equ:lower bound,relative velocity}
|v-u|
\geq&\frac{|e_{i_{1},j_{1}}-e_{i_{2},j_{2}}|}{4}.
\end{align}

Consequently, combining \eqref{equ:E,nontrivial bound} and \eqref{equ:lower bound,relative velocity}, we obtain a nontrivial upper bound
\begin{align}\label{equ:upper bound,Ek}
I_{i_{1},j_{1}}(v) I_{i_{2},j_{2}}(u)\n{E_{i_{1},j_{1}}(t,x,\omega,u,v)}_{L_{x}^{\infty}} \lesssim &\frac{I_{i_{1},j_{1}}(v) I_{i_{2},j_{2}}(u)}{M  \sin \phi_{i_{1},j_{1}}(\omega) \cos \theta(v,u,\omega) |e_{i_{1},j_{1}}-e_{i_{2},j_{2}}|}.
\end{align}

With the estimates for $E_{i_{1},j_{1}}$ established, we move on to analyze the angular integration. For this purpose, we partition the sphere $\mathbb{S}^{2}$ by introducing an angular decomposition. Specifically, we define
\begin{align*}
\Omega=\lr{\omega\in \mathbb{S}^{2}:\sin \phi_{i_{1},j_{1}}(\omega)\leq M^{-1}}\bigcup \lr{\omega\in \mathbb{S}^{2}:\abs{\cos \theta(v,u,\omega)}\leq M^{-\frac{1}{2}}},
\end{align*}
and denote by $\Omega^{c}$ the complementary set of $\Omega$. Accordingly, the integral can be split into two components
\begin{align*}
I_{i_{1},j_{1}}(v) I_{i_{2},j_{2}}(u)\int_{\mathbb{S}^{2}} \n{E_{i_{1},j_{1}}(t,x,\omega,u,v)}_{L_{x}^{\infty}}d\omega
\leq A_{1}+A_{2},
\end{align*}
where
\begin{align*}
A_{1}=&I_{i_{1},j_{1}}(v) I_{i_{2},j_{2}}(u)\int_{\Omega}\n{E_{i_{1},j_{1}}(t,x,\omega,u,v)}_{L_{x}^{\infty}}d\omega,\\
 A_{2}=&I_{i_{1},j_{1}}(v) I_{i_{2},j_{2}}(u)\int_{\Omega^{c}}\n{E_{i_{1},j_{1}}(t,x,\omega,u,v)}_{L_{x}^{\infty}}d\omega.
\end{align*}

We first estimate the term $A_{1}$. By applying the trivial bound \eqref{equ:E,trivial bound} on the set $\Omega$ and utilizing spherical coordinates, we obtain
\begin{align}\label{equ:Eij,A1,estimate}
&\int_{\Omega}\n{E_{i_{1},j_{1}}(t,x,\omega,u,v)}_{L_{x}^{\infty}}d\omega\\
 \leq& \int_{\lr{\omega\in \mathbb{S}^{2}:\sin \phi_{i_{1},j_{1}}(\omega)\leq M^{-1}}}1d\omega+\int_{\lr{\omega\in \mathbb{S}^{2}:\abs{\cos \theta(v,u,\omega)}\leq M^{-\frac{1}{2}}}}1d\omega\notag\\
\lesssim& \int_{0}^{\pi}1_{\lr{\sin \phi\leq M^{-1}}}\sin \phi d\phi+
\int_{0}^{\pi}1_{\lr{|\cos\theta|\leq M^{-\frac{1}{2}}}}\sin \theta d\theta\notag\\
\lesssim& M^{-2}+M^{-\frac{1}{2}}\lesssim M^{-\frac{1}{2}}.\notag
\end{align}

Next, we turn to the estimate for $A_{2}$. Applying the nontrivial upper bound \eqref{equ:upper bound,Ek} yields
\begin{align}
& I_{i_{1},j_{1}}(v)I_{i_{2},j_{2}}(u)
\int_{\Omega^{c}}\n{E_{i_{1},j_{1}}(t,x,\omega,u,v)}_{L_{x}^{\infty}}d\omega\notag\\
\leq&\frac{I_{i_{1},j_{1}}(v)I_{i_{2},j_{2}}(u)}{M|e_{i_{1},j_{1}}-e_{i_{2},j_{2}}|}\int_{\Omega^{c}} \frac{1}{\sin \phi_{i_{1},j_{1}}(\omega)\cos \theta(v,u,\omega)}d\omega\notag\\
\lesssim & \frac{I_{i_{1},j_{1}}(v)I_{i_{2},j_{2}}(u)}{M |e_{i_{1},j_{1}}-e_{i_{2},j_{2}}|}\lrc{\int_{\Omega^{c}}\frac{1}{(\sin \phi_{i_{1},j_{1}}(\omega))^{2}}d\omega+\int_{\Omega^{c}}\frac{1}{(\cos \theta(v,u,\omega))^{2}}d\omega}.\label{equ:estimate,Eij,Omegac,two terms}
\end{align}
For the last two terms on the right hand side of \eqref{equ:estimate,Eij,Omegac,two terms}, by surface integral formula, we have
\begin{align}\label{equ:Eij,A2,estimate,sin}
\int_{\Omega^{c}}\frac{1}{(\sin \phi_{i_{1},j_{1}}(\omega))^{2}}d\omega\lesssim &
\int_{0}^{\pi}1_{\lr{\sin \phi\geq M^{-1}}}\frac{\sin \phi}{(\sin \phi)^{2}}d\phi\\
\lesssim& \int_{M^{-1}}^{1}\frac{1}{\phi}d\phi \lesssim \ln M,\notag
\end{align}
and
\begin{align}\label{equ:Eij,A2,estimate,cos}
\int_{\Omega^{c}}\frac{1}{(\cos \theta(v,u,\omega))^{2}}d\omega\lesssim & \int_{0}^{\pi}1_{\lr{|\cos\theta|\geq M^{-\frac{1}{2}}}}\frac{\sin \theta}{(\cos \theta)^{2}} d\theta\\
\lesssim& \int_{M^{-\frac{1}{2}}}^{1}\frac{1}{\la^{2}}d\la\lesssim M^{\frac{1}{2}}.\notag
\end{align}

Putting together \eqref{equ:Sjk,estimate}, \eqref{equ:Eij,A1,estimate}, \eqref{equ:estimate,Eij,Omegac,two terms}, \eqref{equ:Eij,A2,estimate,sin}, \eqref{equ:Eij,A2,estimate,cos}, we arrive at
\begin{align}\label{equ:Sij,estimate,D+}
\int_{\mathbb{S}^{2}}\n{S_{i_{1},j_{1},i_{2},j_{2}}(t_{0},t,x,\omega,u,v)}_{L_{x}^{\infty}}d\omega \lesssim \frac{1}{M^{\frac{1}{2}}|e_{i_{1},j_{1}}-e_{i_{2},j_{2}}|}I_{i_{1},j_{1}}(v)I_{i_{2},j_{2}}(u).
\end{align}
Substituting \eqref{equ:Sij,estimate,D+} back into the estimate \eqref{equ:estimate,D+} for $D^{+}$, we have
\begin{align}\label{equ:D+,Lv1}
&\n{e^{\lra{v}^{2}}D^{+}}_{L_{v}^{1,1}L_{x}^{\infty}}\\
\lesssim & M^{2-2s}\sum_{(i_{1},j_{1})\neq (i_{2},j_{2})}\int_{\R^{3}\times \R^{3}} |u-v|\lrc{
\int_{\mathbb{S}^{2}} \n{S_{i_{1},j_{1},i_{2},j_{2}}(t_{0},t,x,\omega,u,v)}_{L_{x}^{\infty}}d\omega }du dv\notag\\
\lesssim & M^{2-2s}\sum_{(i_{1},j_{1})\neq (i_{2},j_{2})} \frac{1}{M^{\frac{1}{2}}|e_{i_{1},j_{1}}-e_{i_{2},j_{2}}|}
\int_{\R^{3}\times \R^{3}}
|u-v|I_{i_{1},j_{1}}(v)I_{i_{2},j_{2}}(u)du dv,\notag\\
\lesssim & M^{-\frac{5}{2}-2s}\sum_{(i_{1},j_{1})\neq (i_{2},j_{2})} \frac{1}{|e_{i_{1},j_{1}}-e_{i_{2},j_{2}}|},\notag
\end{align}
where in the last inequality we have used that
\begin{align}\label{equ:D+,Lv1,IjIk}
\int_{\R^{3}\times \R^{3}} |u-v|I_{i_{1},j_{1}}(v)I_{i_{2},j_{2}}(u)du dv
\lesssim&\n{I_{i_{1},j_{1}}}_{L_{v}^{1,1}}\n{I_{i_{2},j_{2}}}_{L_{v}^{1,1}}\leq M^{-4}.
\end{align}

It remains to handle the summation estimates over the lattice system. We decompose the sum into two cases based on the indices
\begin{align}\label{equ:sum,eij,twocases}
\sum_{(i_{1},j_{1})\neq (i_{2},j_{2})}   \frac{1}{|e_{i_{1},j_{1}}-e_{i_{2},j_{2}}|}
\leq &\sum_{j_{1}=j_{2},i_{1}\neq i_{2}}\frac{1}{|e_{i_{1},j_{1}}-e_{i_{2},j_{2}}|}+\sum_{j_{1}\neq j_{2},i_{1},i_{2}}\frac{1}{|e_{i_{1},j_{1}}-e_{i_{2},j_{2}}|}.
\end{align}
As established in \eqref{equ:gap estimate,eij,j1neqj2} and \eqref{equ:gap estimate,eij,j1=j2}, we have
\begin{align*}
|e_{i_{1},j_{1}}-e_{i_{2},j_{2}}|\geq &\frac{|j_{1}-j_{2}|}{M},\quad j_{1}\neq j_{2},\\
|e_{i_{1},j}-e_{i_{2},j}|\gtrsim& \frac{1}{M},\quad j_{1}=j_{2}=j,\ i_{1}\neq i_{2}.
\end{align*}
Consequently, for the first term in \eqref{equ:sum,eij,twocases}, we have
\begin{align}\label{equ:sum,eij,i1neqi2}
&\sum_{j=1}^{M}\sum_{j_{1}=j_{2}=j,i_{1}\neq i_{2}}
\frac{1}{|e_{i_{1},j_{1}}-e_{i_{2},j_{2}}|}
\lesssim \sum_{j=1}^{M}\sum_{j_{1}=j_{2}=j,i_{1}\neq i_{2}}
M\lesssim M^{4},
\end{align}
while for the second term in \eqref{equ:sum,eij,twocases}, we obtain
\begin{align}\label{equ:sum,eij,j1neqj2}
\sum_{j_{1}\neq j_{2},i_{1},i_{2}}\frac{1}{|e_{i_{1},j_{1}}-e_{i_{2},j_{2}}|}
\lesssim&\sum_{j_{1}\neq j_{2},i_{1},i_{2}}\frac{M}{|j_{1}-j_{2}|}\\
\lesssim& M^{3}\sum_{j_{1}=1}^{M}\sum_{j_{2}=1}^{M}\frac{1}{|j_{1}-j_{2}|+1}\notag\\
\lesssim& M^{4}\ln M.\notag
\end{align}

Putting \eqref{equ:sum,eij,twocases}, \eqref{equ:sum,eij,i1neqi2}, \eqref{equ:sum,eij,j1neqj2} into \eqref{equ:D+,Lv1}, we arrive at
\begin{align}\label{equ:D+,Lv1,case1}
\n{e^{\lra{v}^{2}}D^{+}}_{L_{v}^{1,1}L_{x}^{\infty}}\lesssim& M^{-\frac{5}{2}-2s} M^{4}\ln M\lesssim M^{\frac{3}{2}-2s}\ln M.
\end{align}

\textbf{Case $2$: $(i_{1},j_{1})=(i_{2},j_{2})$.}

In this diagonal case, the double summation
 $\sum_{i_{1},j_{1}}\sum_{i_{2},j_{2}}$ reduces to order $M^{2}$, gaining an
$M^{-2}$ decay compared to the non-diagonal case. Therefore, it suffices to apply the trivial bound \eqref{equ:E,trivial bound} to get
\begin{align*}
\int_{\mathbb{S}^{2}}\n{S_{i_{1},j_{1},i_{2},j_{2}}(t_{0},t,x,\omega,u,v)}_{L_{x}^{\infty}}d\omega\lesssim I_{i_{1},j_{1}}(v)I_{i_{2},j_{2}}(u).
\end{align*}
Together with the estimate \eqref{equ:estimate,D+} for $D^{+}$, it follows that
\begin{align}\label{equ:Q+,fb,fb,est2,final,case2}
&\n{e^{\lra{v}^{2}}D^{+}}_{L_{v}^{1,1}L_{x}^{\infty}}\\
\lesssim&  M^{2-2s}\sum_{(i_{1},j_{1})=(i_{2},j_{2})} \int_{\R^{3}\times \R^{3}} |u-v|\lrc{\int_{\mathbb{S}^{2}}
\n{S_{i_{1},j_{1},i_{2},j_{2}}(t_{0},t,x,\omega,u,v)}_{L_{x}^{\infty}}d\omega}du
dv\notag\\
\lesssim&  M^{2-2s}\sum_{(i_{1},j_{1})=(i_{2},j_{2})}
\int_{\R^{3}\times \R^{3}}
|u-v|I_{i_{1},j_{1}}(v)I_{i_{2},j_{2}}(u)dudv\notag\\
\lesssim&  M^{2-2s}\sum_{(i_{1},j_{1})=(i_{2},j_{2})}
\n{I_{i_{2},j_{2}}}_{L_{v}^{1,1}}\n{I_{i_{1},j_{1}}}_{L_{v}^{1,1}}\notag\\
\lesssim&  M^{2-2s} M^{2} M^{-4}\notag\\
\lesssim& M^{-2s} .\notag
\end{align}

Combining the estimates \eqref{equ:D+,Lv1,case1} and \eqref{equ:Q+,fb,fb,est2,final,case2} for the two cases, we finally arrive at
\begin{align*}
\n{e^{\lra{v}^{2}}D^{+}}_{L_{v}^{1,1}L_{x}^{\infty}}\lesssim M^{\frac{3}{2}-2s}\ln M,
\end{align*}
which is sufficient for our purpose provided that $s>\frac{3}{4}$.

\end{proof}

\section{Stability Analysis for the Approximate Solutions via an Iterative Scheme}\label{section:Stability Analysis for the Approximate Solutions via an Iterative Scheme}
In this section, we establish the stability of the constructed approximate solutions up to the critical time $T_{*}$, at which point the ill-posedness phenomenon manifests.
The main task is to provide uniform-in-time decay estimates for the correction terms $f_{\mathrm{c}}(t)$ and $g_{\mathrm{c}}(t)$,
which are governed by the following equations
\begin{equation}\label{equ:correction term,fc,stability}
\left\{
\begin{aligned} \partial_t f_{\mathrm{c}} + v\cdot \nabla_x f_{\mathrm{c}} = & \pm Q^\pm(f_{\mathrm{c}},f_{ \mathrm{a}}) \pm
Q^\pm(f_{\mathrm{a}},f_{\mathrm{c}}) \pm Q^\pm(f_{\mathrm{c}},f_{\mathrm{c}
})-F_{\mathrm{err}},\\
F_{\mathrm{err}}=&\pa_{t}f_{\mathrm{a}}+v\cdot \nabla_{x}f_{\mathrm{a}}+Q^{-}(f_{\mathrm{a}},f_{\mathrm{a}})-Q^{+}(f_{\mathrm{a}},f_{\mathrm{a}}),\\
f_{\mathrm{c}}(0)=&0,
 \end{aligned}
 \right.
\end{equation}
and
\begin{equation}\label{equ:correction term,gc}
\left\{
\begin{aligned} \partial_t g_{\mathrm{c}} + v\cdot \nabla_x g_{\mathrm{c}} = & \pm Q^\pm(g_{\mathrm{c}},g_{ \mathrm{a}}) \pm
Q^\pm(g_{\mathrm{a}},g_{\mathrm{c}}) \pm Q^\pm(g_{\mathrm{c}},g_{\mathrm{c}
})-G_{\mathrm{err}},\\
G_{\mathrm{err}}=&\pa_{t}g_{\mathrm{a}}+v\cdot \nabla_{x}g_{\mathrm{a}}+Q^{-}(g_{\mathrm{a}},g_{\mathrm{a}})-Q^{+}(g_{\mathrm{a}},g_{\mathrm{a}}),\\
g_{\mathrm{c}}(0)=&0.
 \end{aligned}
 \right.
\end{equation}
Here, the approximate solution $g_{\mathrm{a}}(t)$ is taken to be the stationary state $f_{\mathrm{r}}(0)$, that is,
$$g_{\mathrm{a}}(t)\equiv f_{\mathrm{r}}(0).$$
This choice ensures that the two exact solutions $f(t)$ and $g(t)$ are initially close in the strong weighted norm.
Moreover, the norm deflation property previously established for $f_{\mathrm{a}}(t)$, together with the uniform decay estimates for the correction terms $f_{\mathrm{c}}(t)$ and $g_{\mathrm{c}}(t)$, then guarantees that the two solutions are fully separated at the critical time $T_{*}$.

To ensure the uniform control of $f_{\mathrm{c}}(t)$ and $g_{\mathrm{c}}(t)$
over the interval $[0,T_{*}]$,
we employ an iterative scheme that relies on space-time collision estimates in the Gaussian-weighted $Z$-norm established in Section \ref{section:Space-time Collision}, the $Z$-norm estimate for $f_{\mathrm{a}}(t)$, and the uniform-in-time decay estimates for the error term $F_{\mathrm{err}}(t)$ in Section \ref{section:Estimates for Error Terms}.

\begin{proposition}\label{lemma:perturbation}
Let $T_{*}=M^{s-1}(\ln\ln M)$ and $\kappa=\frac{1}{T_{*}}$.
 Then, we have
\begin{align}
		\sup_{t\in[0,T_{*}]}\n{e^{(1-\kappa t)\lra{v}^{2}}f_{\mathrm{c}}(t)}_{Z}\lesssim& M^{-\frac{\delta}{2}},  \label{E:fc_bound2}\\
		\sup_{t\in[0,T_{*}]}\n{e^{(1-\kappa t)\lra{v}^{2}}g_{\mathrm{c}}(t)}_{Z}\lesssim& M^{-\frac{\delta}{2}}.  \label{E:gc_bound2}
	\end{align}
In particular, we have
\begin{align}
		\sup_{t\in[0,T_{*}]}\n{e^{(1-\kappa t)\lra{v}^{2}}f_{\mathrm{c}}(t)}_{L_{v}^{2} H_{x}^{s_{0}}}\lesssim& M^{-\frac{\delta}{2}},  \label{E:fc_bound2,sobolev}\\
		\sup_{t\in[0,T_{*}]}\n{e^{(1-\kappa t)\lra{v}^{2}}g_{\mathrm{c}}(t)}_{L_{v}^{2} H_{x}^{s_{0}}}\lesssim& M^{-\frac{\delta}{2}}.  \label{E:gc_bound2,sobolev}
	\end{align}
\end{proposition}

\begin{proof}
It suffices to deal with the $Z$-norm estimates \eqref{E:fc_bound2}--\eqref{E:gc_bound2}.
Indeed, the corresponding Sobolev norm estimates \eqref{E:fc_bound2,sobolev}--\eqref{E:gc_bound2,sobolev} are a direct consequence of a standard interpolation argument. Specifically, we have
\begin{align*}
\n{f}_{L_{v}^{2}H_{x}^{s_{0}}}\leq &\n{f}_{L_{v}^{2}L_{x}^{2}}^{1-s_{0}}\n{\nabla_{x}f}_{L_{v}^{2}L_{x}^{2}}^{s_{0}}\\
\leq&
\lrs{M^{s_{0}}\n{f}_{L_{v}^{2}L_{x}^{2}}}^{1-s_{0}} \lrs{M^{s_{0}-1}\n{\nabla_{x}f}_{L_{v}^{2}L_{x}^{2}}}^{s_{0}}\\
\leq&\n{f}_{Z}.
\end{align*}

We first prove the estimate \eqref{E:fc_bound2} for $f_{\mathrm{c}}(t)$ and the argument for $g_{\mathrm{c}}(t)$ is analogous and will be summarized at the end.
We partition the time interval $[0,T_{*}]$ as
	\begin{equation*}
		0=T_{0}<T_{1}<T_{2}<\cdots <T_{n-1}<T_{n}=T_{*}=M^{s-1}\ln\ln M
	\end{equation*}
	where
$$T_{j}= \frac{jM^{s-1}}{\sqrt{\ln M}},\quad n= \sqrt{\ln M} (\ln\ln M).$$
Consequently, the length of each subinterval
 $I_{j}=[T_{j},T_{j+1}]$ is given by
	\begin{equation*}
		|I_{j}|=  \frac{M^{s-1}}{\sqrt{\ln M}}.
	\end{equation*}

	For $t\in I_{j}=[T_{j},T_{j+1}]$, we express the solution $f_{\mathrm{c}}(t)$ to equation \eqref{equ:correction term,fc,stability} in its Duhamel formulation
\begin{align*}
f_{\mathrm{c}}(t)=e^{-(t-T_{j})v\cdot \nabla_{x}}f_{\mathrm{c}}(T_{j})+\int_{T_{j}}^{t}e^{-(t-\tau)v\cdot \nabla_{x}}F(\tau)d\tau,
\end{align*}
where the forcing term
$F$ is defined as
$$F=\pm Q^\pm(f_{\mathrm{c}},f_{ \mathrm{a}}) \pm
Q^\pm(f_{\mathrm{a}},f_{\mathrm{c}}) \pm Q^\pm(f_{\mathrm{c}},f_{\mathrm{c}
})-F_{\mathrm{err}}.$$

	Applying the
space-time collision estimates in the Gaussian-weighted $Z$-norm from Proposition \ref{lemma:Z,norm,bilinear,estimate}, we obtain
	\begin{align}\label{equ:fc,Z-norm,gauss}
		&\n{e^{(1-\kappa t)\lra{v}^{2}}f_{\mathrm{c}}}_{L_{I_{j}}^{\infty }Z}\\
 \leq& \n{e^{(1-\kappa T_{j})\lra{v}^{2}}f_{\mathrm{c}}(T_{j})}_{Z}
+\bbn{e^{(1-\kappa t)\lra{v}^{2}}\int_{T_{j}}^{t}e^{-(t-\tau)v\cdot \nabla_{x}}F(\tau)d\tau}_{L_{t}^{\infty}(I_{j};Z)}\notag \\
		 \lesssim & \n{e^{(1-\kappa T_{j})\lra{v}^{2}}f_{\mathrm{c}}(T_{j})}_{Z}+ \lrs{\frac{|I_{j}|}{\kappa}}^{\frac{1}{2}}
\n{e^{(1-\kappa t)\lra{v}^{2}}f_{\mathrm{a}}}_{L_{t}^{\infty}(I_{j};Z)}  \n{e^{(1-\kappa t)\lra{v}^{2}}f_{\mathrm{c}}}_{L_{t}^{\infty}(I_{j};Z)}\notag\\
&+\lrs{\frac{|I_{j}|}{\kappa}}^{\frac{1}{2}}\n{e^{(1-\kappa t)\lra{v}^{2}}f_{\mathrm{c}}}_{L_{t}^{\infty}(I_{j};Z)}^{2}
		 +\bbn{e^{(1-\kappa t)\lra{v}^{2}}\int_{T_{j}}^{t}e^{-(t-\tau)v\cdot \nabla_{x}}F_{\mathrm{err}}(\tau)d\tau}_{L_{t}^{\infty}(I_{j};Z)}.\notag
	\end{align}
In view of the parameter choice $\frac{1}{\kappa}=T_{*}=M^{s-1}\ln\ln M$, the coefficient in \eqref{equ:fc,Z-norm,gauss} is bounded by
\begin{align}\label{equ:Ij,kappa}
\lrs{\frac{|I_{j}|}{\kappa}}^{\frac{1}{2}}\lesssim \frac{M^{s-1}(\ln \ln M)^{\frac{1}{2}}}{(\ln M)^{\frac{1}{4}}} \lesssim \frac{M^{s-1}}{(\ln M)^{\frac{1}{10}}}.
\end{align}

Regarding the term involving $f_{\mathrm{a}}$, Lemma \ref{lemma:z-norm bounds on,fb,fr,fa} provides the $Z$-norm bound
\begin{align}\label{equ:z-norm estimate for fa,fc,proof}
\n{e^{\lra{v}^{2}}f_{\mathrm{a}}(t)}_{L_{t}^{\infty}(0,T_{*};Z)}\lesssim M^{1-s}.
\end{align}

For the error term $F_{\mathrm{err}}$, we employ the uniform-in-time decay estimate \eqref{equ:Ferr_bound} from Proposition \ref{lemma:bounds on ferr}, which gives
\begin{equation}\label{equ:Ferr_bound,fc,proof}
\bbn{e^{\lra{v}^{2}}\int_{t_{0}}^{t}e^{-(t-\tau)v\cdot \nabla _{x}}F_{\mathrm{err}
}(\tau)\,d\tau}_{Z}\lesssim M^{-\delta},\quad 0\leq t_{0}\leq t\leq T_{*}.
\end{equation}

Inserting the estimates \eqref{equ:Ij,kappa}, \eqref{equ:z-norm estimate for fa,fc,proof}, and \eqref{equ:Ferr_bound,fc,proof} into \eqref{equ:fc,Z-norm,gauss}, we obtain
	\begin{align*}
		\n{e^{(1-\kappa t)\lra{v}^{2}}f_{\mathrm{c}}}_{L_{t}^{\infty}(I_{j};Z)}\leq& \n{e^{(1-\kappa T_{j})\lra{v}^{2}}f_{\mathrm{c}}(T_{j})}_{Z}
+ \frac{C}{(\ln M)^{\frac{1}{10}}}\n{e^{(1-\kappa t)\lra{v}^{2}}f_{\mathrm{c}}}_{L_{t}^{\infty}(I_{j};Z)}\\
&+ \frac{CM^{s-1}}{(\ln M)^{\frac{1}{10}}}\n{e^{(1-\kappa t)\lra{v}^{2}} f_{\mathrm{c}}}_{L_{I_{j}}^{\infty
			}Z}^{2}+CM^{-\delta},
	\end{align*}
where $C$ denotes an absolute constant. By absorbing the linear term $\Vert f_{\mathrm{c}}\Vert
_{L_{I_{j}}^{\infty }Z}$ on the right-hand side, we deduce
\begin{equation*}
\n{e^{(1-\kappa t)\lra{v}^{2}}f_{\mathrm{c}}}_{L_{t}^{\infty}(I_{j};Z)}\leq 2\n{e^{(1-\kappa T_{j})\lra{v}^{2}}f_{\mathrm{c}
}(T_{j})}_{Z}+2CM^{-\delta}.
\end{equation*}

We now iterate this estimate from $j=0$ to $j=n-1$. Given the initial condition $f_{\mathrm{c}}(0)=0$, we have for $j=0$
\begin{equation*}
\n{e^{(1-\kappa t)\lra{v}^{2}}f_{\mathrm{c}}}_{L_{t}^{\infty}(I_{j};Z)}\leq 2CM^{-\delta}.
\end{equation*}
Proceeding inductively for $j=1,\ldots $, we obtain
\begin{align*}
\n{e^{(1-\kappa t)\lra{v}^{2}}f_{\mathrm{c}}}_{L_{t}^{\infty}(I_{j};Z)}\leq&(2^{j+1}-1)2CM^{-\delta}.
\end{align*}

Since $j+1=n=\sqrt{\ln M}(\ln\ln M)$, we conclude
\begin{equation*}
\sup_{t\in[0,T_{*}]}\Vert e^{(1-\kappa t)\lra{v}^{2}}f_{\mathrm{c}}(t)\Vert _{Z}\leq \frac{Ce^{\sqrt{\ln M}\ln \ln M}}{M^{\delta}} = \frac{Ce^{\sqrt{\ln M}\ln \ln M}}{e^{\delta \ln M}}\leq
M^{-\frac{\delta}{2}}\ll 1.
\end{equation*}

The proof for $g_{\mathrm{c}}(t)$ follows in an analogous manner. We express $g_{\mathrm{c}}(t)$ via the Duhamel formula
\begin{align*}
g_{\mathrm{c}}(t)=e^{-(t-T_{j})v\cdot \nabla_{x}}g_{\mathrm{c}}(T_{j})+\int_{T_{j}}^{t}e^{-(t-\tau)v\cdot \nabla_{x}}G(\tau)d\tau,
\end{align*}
where the forcing term
$G(t)$ is given by
\begin{equation*}
\left\{
\begin{aligned}
G=&\pm Q^\pm(g_{\mathrm{c}},g_{ \mathrm{a}}) \pm
Q^\pm(g_{\mathrm{a}},g_{\mathrm{c}}) \pm Q^\pm(g_{\mathrm{c}},g_{\mathrm{c}
})-G_{\mathrm{err}},\\
G_{\mathrm{err}}=&\pa_{t}g_{\mathrm{a}}+v\cdot \nabla_{x}g_{\mathrm{a}}+Q^{-}(g_{\mathrm{a}},g_{\mathrm{a}})-Q^{+}(g_{\mathrm{a}},g_{\mathrm{a}}).
\end{aligned}
\right.
\end{equation*}

Adopting the iterative scheme employed for $f_{\mathrm{c}}(t)$, it remains to establish the counterparts of the key estimates \eqref{equ:z-norm estimate for fa,fc,proof} for $g_{\mathrm{a}}(t)$ and \eqref{equ:Ferr_bound,fc,proof} for $G_{\mathrm{err}}(t)$. Observing that $g_{\mathrm{a}}(t)\equiv f_{\mathrm{r}}(0)$, we employ \eqref{equ:z-norm bounds for fr} from Lemma \ref{lemma:z-norm bounds on,fb,fr,fa} to obtain
\begin{align}\label{equ:z-norm estimate for ga,gc,proof}
\n{e^{\lra{v}^{2}}g_{\mathrm{a}}(t)}_{L_{t}^{\infty}(0,T_{*};Z)}=\n{e^{\lra{v}^{2}}f_{\mathrm{r}}(0)}_{Z}\lesssim 1\lesssim M^{1-s}.
\end{align}

For the error term $G_{\mathrm{err}}(t)$,  repeating the analysis of Sections \ref{section:Analysis of term1}--\ref{section:Analysis of term2} originally applied to the term $f_{\mathrm{r}}(t)$, yields, for $0\leq t_{0}\leq t\leq T_{*}$,
\begin{align}\label{equ:Ferr_bound,gc,proof}
&\bbn{e^{\lra{v}^{2}}\int_{t_{0}}^{t}e^{-(t-\tau)v\cdot \nabla _{x}}G_{\mathrm{err}
}(\tau)\,d\tau}_{Z}\\
=&\bbn{e^{\lra{v}^{2}}\int_{t_{0}}^{t}e^{-(t-\tau)v\cdot \nabla _{x}}\lrs{v\cdot \nabla_{x}g_{\mathrm{a}}+Q^{-}(g_{\mathrm{a}},g_{\mathrm{a}})-Q^{+}(g_{\mathrm{a}},g_{\mathrm{a}})(\tau)}d\tau}_{Z}\notag\\
=&\bbn{e^{\lra{v}^{2}}\int_{t_{0}}^{t}e^{-(t-\tau)v\cdot \nabla _{x}}\lrs{v\cdot \nabla_{x}f_{\mathrm{r}}(0)+Q^{-}(f_{\mathrm{r}}(0),f_{\mathrm{r}}(0))-Q^{+}(f_{\mathrm{r}}(0),f_{\mathrm{r}}(0))}d\tau}_{Z}\notag\\
\lesssim& M^{-\delta}.\notag
\end{align}
With \eqref{equ:z-norm estimate for ga,gc,proof} and \eqref{equ:Ferr_bound,gc,proof} in hand, the same iterative argument applied to $g_{\mathrm{c}}(t)$ leads to the desired estimate \eqref{E:gc_bound2}. This completes the proof.

\end{proof}

\section{Proof of Ill-posedness}\label{section:Proof of Ill-posedness}
We are now in a position to present the proof of the ill-posedness result stated in Theorem \ref{thm:main theorem}.
The strategy relies on constructing a family of smooth initial data pairs $(f(0),g(0))$ that are arbitrarily close, yet for which the corresponding solutions at a prescribed final time $T_{*}$
remain separated by a positive distance.
The proof is built upon three key estimates, the Sobolev norm estimates for $f_{\mathrm{b}}(t)$ in Lemma \ref{lemma:Hs,bounds on fb}, Sobolev norm estimates for $f_{\mathrm{r}}(t)$ from Lemma \ref{lemma:bounds on fr}, and uniform-in-time decay estimates for the correction terms $f_{\mathrm{c}}(t)$ and $g_{\mathrm{c}}(t)$ provided in Proposition \ref{lemma:perturbation}.

\begin{proof}[\textbf{Proof of ill-posedness in Theorem $\ref{thm:main theorem}$}]
We consider the pair of solutions constructed in the previous Sections \ref{sec:Bounds on the approximation solution}--\ref{section:Stability Analysis for the Approximate Solutions via an Iterative Scheme}. These solutions admit the following decompositions:
\begin{equation*}
\left\{
\begin{aligned}
f(t) &= f_{\mathrm{r}}(t) + f_{\mathrm{b}}(t) + f_{\mathrm{c}}(t), \quad f_{\mathrm{c}}(0)=0,\\
g(t) &= f_{\mathrm{r}}(0) + g_{\mathrm{c}}(t),\quad g_{\mathrm{c}}(0)=0.
\end{aligned}
\right.
\end{equation*}
Consequently, the difference between these two solutions can be explicitly expressed as
\begin{equation*}
f(t) - g(t) = f_{\mathrm{r}}(t) - f_{\mathrm{r}}(0) + f_{\mathrm{b}}(t) + f_{\mathrm{c}}(t) - g_{\mathrm{c}}(t).
\end{equation*}
By applying Lemma \ref{lemma:Hs,bounds on fb}, Lemma \ref{lemma:bounds on fr}, and Proposition \ref{lemma:perturbation}, we have
\begin{equation}\label{equ:fb,fr,fc,proof}
\left\{
\begin{aligned}
\sup_{t \in [0, T_{*}]} \n{e^{\lra{v}^{2}} f_{\mathrm{b}}(t)}_{L_{v}^{2} H_{x}^{s_{0}}} &\lesssim M^{s_{0}-s}\lesssim \frac{1}{\ln M}, \\
\n{f_{\mathrm{r}}(0)}_{L_{v}^{2} H_{x}^{s_{0}}} &\gtrsim 1,\\
\n{e^{\lra{v}^{2}} f_{\mathrm{r}}(T_{*})}_{L_{v}^{2} H_{x}^{s_{0}}} &\lesssim \frac{1}{\ln M}, \\
\sup_{t \in [0, T_{*}]} \n{e^{(1-\kappa t) \lra{v}^{2}} f_{\mathrm{c}}(t)}_{L_{v}^{2} H_{x}^{s_{0}}}  \lesssim& M^{-\frac{\delta}{2}}\lesssim \frac{1}{\ln M}, \\
\sup_{t \in [0, T_{*}]} \n{e^{(1-\kappa t) \lra{v}^{2}} g_{\mathrm{c}}(t)}_{L_{v}^{2} H_{x}^{s_{0}}}  \lesssim& M^{-\frac{\delta}{2}}\lesssim \frac{1}{\ln M},
\end{aligned}
\right.
\end{equation}
where the final time $T_{*}=M^{s-1}(\ln\ln M)$. Under the choice
$s<1$ from \eqref{equ:condition,s,s0}, this time scale becomes arbitrarily short in the limit, i.e.,
\begin{align*}
\lim_{M\to \infty}T_{*}=0.
\end{align*}

We first examine the difference at the initial time. Since $f_{\mathrm{c}}(0)=g_{\mathrm{c}}(0)=0$ by construction, the difference $f(0)-g(0)$ reduces simply to $f_{\mathrm{b}}(0)$. Applying the first estimate in \eqref{equ:fb,fr,fc,proof} at $t=0$, we obtain
\begin{equation*}
\n{e^{\lra{v}^{2}} \lrs{f(0) - g(0)}}_{L_{v}^{2} H_{x}^{s_{0}}}
= \n{e^{\lra{v}^{2}} f_{\mathrm{b}}(0)}_{L_{v}^{2} H_{x}^{s_{0}}} \lesssim \frac{1}{\ln M}.
\end{equation*}
 This confirms that the two initial data can be made arbitrarily close in the weighted $L_{v}^{2} H_{x}^{s_{0}}$ norm by choosing $M$ sufficiently large.

We now analyze the difference between the solutions at the prescribed final time $T_{*}$. By construction, the parameters satisfy $\kappa T_{*} = 1$, which implies that the exponential weight
becomes trivial at $t=T_{*}$, that is,
\begin{align}\label{equ:kappa,T,1}
e^{(1-\kappa T_{*})\lra{v}^{2}}=1.
\end{align}
Thus, at the final time $T_{*}$, the weighted norm coincides exactly with the standard unweighted $L_{v}^{2} H_{x}^{s_{0}}$ norm.
Starting from the identity
\begin{equation*}
f(T_{*}) - g(T_{*}) = f_{\mathrm{r}}(T_{*}) - f_{\mathrm{r}}(0) + f_{\mathrm{b}}(T_{*}) + f_{\mathrm{c}}(T_{*}) - g_{\mathrm{c}}(T_{*}),
\end{equation*}
we take the $L_{v}^{2}H_{x}^{s_{0}}$ norm, invoke \eqref{equ:kappa,T,1} at $t=T_{*}$, and apply the triangle inequality to deduce
\begin{align}\label{equ:proof,T,difference}
&\n{f(T_{*}) - g(T_{*})}_{L_{v}^{2} H_{x}^{s_{0}}}\\
=& \n{e^{(1-\kappa T_{*})\lra{v}^{2}} \lrs{f(T_{*}) - g(T_{*})}}_{L_{v}^{2} H_{x}^{s_{0}}} \notag\\
\ge& \n{e^{(1-\kappa T_{*})\lra{v}^{2}} f_{\mathrm{r}}(0)}_{L_{v}^{2} H_{x}^{s_{0}}}
- \n{e^{(1-\kappa T_{*})\lra{v}^{2}} f_{\mathrm{r}}(T_{*})}_{L_{v}^{2} H_{x}^{s_{0}}}
- \n{e^{(1-\kappa T_{*})\lra{v}^{2}} f_{\mathrm{b}}(T_{*})}_{L_{v}^{2} H_{x}^{s_{0}}} \notag\\
&\quad - \n{e^{(1-\kappa T_{*})\lra{v}^{2}} f_{\mathrm{c}}(T_{*})}_{L_{v}^{2} H_{x}^{s_{0}}}
- \n{e^{(1-\kappa T_{*})\lra{v}^{2}} g_{\mathrm{c}}(T_{*})}_{L_{v}^{2} H_{x}^{s_{0}}}.\notag
\end{align}
Substitute the bounds from \eqref{equ:fb,fr,fc,proof} into the above inequality \eqref{equ:proof,T,difference}, we arrive at the lower bound
\begin{align*}
\n{f(T_{*}) - g(T_{*})}_{L_{v}^{2} H_{x}^{s_{0}}}
\gtrsim& \n{f_{\mathrm{r}}(0)}_{L_{v}^{2} H_{x}^{s_{0}}} - \frac{1}{\ln M}
\gtrsim 1.
\end{align*}
This establishes the desired separation at time $T_{*}$, thereby completing the proof of the theorem.
\end{proof}

\noindent \textbf{Acknowledgements}
The authors would like to thank Professor Kenji Nakanishi and Professor Tong Yang for raising the problem.
S. Shen was supported in part by NSF of China under Grant 12501322.
Z. Zhang was supported in part by NSF of China under Grant 12288101.

\noindent\textbf{Data Availability Statement}
Data sharing is not applicable to this article as no datasets were generated or analysed during the current study.

\noindent\textbf{Conflict of Interest}
The authors declare that they have no conflict of interest.

\appendix
\section{Collision Estimates}
\begin{lemma}
For the hard-sphere kernel, it holds that
\begin{align}
\n{Q^{-}(f,g)}_{L_{v}^{r}}\lesssim& \n{f}_{L_{v}^{r,1}}\n{g}_{L_{v}^{1,1}},\quad \text{for $r\in[1,\infty]$},\label{equ:collision,estimate,Q-,Lr}\\
\n{Q^{+}(f,g)}_{L_{v}^{r}}\lesssim& \n{f}_{L_{v}^{r,1}}\n{g}_{L_{v}^{1,1}},\quad \text{for $r\in[1,\infty]$},\label{equ:collision,estimate,Q+,Lr}\\
\n{Q^{+}(f,g)}_{L_{v}^{r}}\lesssim& \n{f}_{L_{v}^{1,1}}\n{g}_{L_{v}^{r,1}},\quad \text{for $r\in[2,\infty]$}.\label{equ:collision,estimate,Q+,Lr,gr}
\end{align}
\end{lemma}
\begin{proof}
The estimate for the loss term \eqref{equ:collision,estimate,Q-,Lr} follows directly. Indeed, by bounding the relative velocity as $|u-v|\lesssim
\lra{u}\lra{v}$, we obtain
 \begin{align*}
\n{Q^{-}(f,g)}_{L_{v}^{r}}\lesssim& \bbn{f(v)\int_{\R^{3}}|u-v||g(u)|du}_{L_{v}^{r}}\\
\lesssim&\n{\lra{v}f(v)}_{L_{v}^{r}}\n{\lra{u}g(u)}_{L_{u}^{1}}.
 \end{align*}
 The gain term estimates \eqref{equ:collision,estimate,Q+,Lr}--\eqref{equ:collision,estimate,Q+,Lr,gr} follow from \cite[Theorem 1]{ACG10} or \cite[Theorem 2.1]{MV04}.
\end{proof}

\bibliographystyle{abbrv}
\bibliography{references}

\end{document}